\documentclass[12pt]{article}

\usepackage{ulem}
\usepackage[margin=1in]{geometry}
\usepackage{amscd}
\usepackage{amssymb}
\usepackage{amsmath}
\usepackage{amsthm}
\usepackage{enumerate}
\usepackage{tikz}
\usepackage{wrapfig, framed, caption}
\usepackage{float}
\usetikzlibrary{arrows}
\usepackage[font=small,labelfont=bf]{caption}
\usepackage{xcolor}
\usepackage{mathtools}
\usepackage[colorlinks=true, linkcolor=blue, citecolor=blue,
pagebackref=true]{hyperref}
\usepackage{fouriernc}
\usepackage{ulem}
\allowdisplaybreaks

{\theoremstyle{definition}
\newtheorem{definition}{Definition}[section]} 
\newtheorem{theorem}[definition]{Theorem}
\newenvironment{theorem*}[1]{{\bf Theorem #1} \begin{itshape}}{\end{itshape}}
\newtheorem{lemma}[definition]{Lemma}
\newtheorem{corollary}[definition]{Corollary}
\newenvironment{corollary*}[1]{{\bf Corollary #1} \begin{itshape}}{\end{itshape}}
\newtheorem{proposition}[definition]{Proposition}
\newenvironment{proposition*}[1]{{\bf Proposition #1} \begin{itshape}}{\end{itshape}}

\theoremstyle{remark}
\newtheorem{remark}[definition]{Remark}

\numberwithin{equation}{section}

\renewcommand{\Bbb}[1]{\mathbb{#1}}

\newcommand{\N}{{\Bbb N}}         

\newcommand{\R}{{\Bbb R}}        

\newcommand{\cB}{{\mathcal B}}

\newcommand{\cH}{{\mathcal H}}

\newcommand{\cK}{{\mathcal K}}

\newcommand{\bt}{\mathbf{t}}

\DeclareMathOperator{\dimh}{\dim_H}

\title{\sc Rectangular Shrinking Targets on Self-Similar Carpets}

\author{Demi Allen \\ (Exeter)\\ email: \href{mailto:d.d.allen@exeter.ac.uk}{d.d.allen@exter.ac.uk} \and Thomas Jordan \\ (Bristol)\\ email: \href{mailto:thomas.jordan@bristol.ac.uk}{thomas.jordan@bristol.ac.uk} \and Benjamin Ward \\ (York)\\email:\href{mailto:benjamin.ward@york.ac.uk}{benjamin.ward@york.ac.uk}}
\date{\today}
\begin{document}
\frenchspacing
\maketitle
\begin{abstract}
Since the introduction of the shrinking target problem by Hill and Velani in 1995 there has been a surge of interest in the area. In this paper we consider the case where the target is a rectangle, rather than a ball, and the underlying space is a self-similar carpet. We calculate the exact Hausdorff dimension of the resulting shrinking target set. Interestingly the Hausdorff dimension depends on the centre of the target, a condition uncommon in most other shrinking target type problems. This extends a theorem of Wang and Wu [Theorem 12.1, Math. Ann. 2021].  

\end{abstract}
%


\maketitle

\section{Introduction}

Let $(X,d)$ be a metric space equipped with a Borel probability measure $\mu$ and let \mbox{$T: X \to X$} be a measure preserving transformation; that is, for any Borel set $A \subset X$, we have \linebreak \mbox{$\mu(T^{-1}A) = \mu(A)$}. 
Suppose we are given a sequence $(B_n)_{n \in \N}$ of measurable sets in $X$, and suppose that these sets are ``shrinking'' in some way. Then, the classical \emph{shrinking target problem} is concerned with studying the set of points $x \in X$ which ``hit'' infinitely many of the targets, $B_n$, when acted upon by $T$; i.e., one is often interested in the set
\[\cB = \left\{x \in X: T^n x \in B_n \quad \text{for i.m. } n \in \N\right\},\]
where `i.m.' denotes `infinitely many'. 
Typically one is interested in studying the measure theoretic properties of the set $\cB$. However, initiated by the work of Hill and Velani \cite{HillVelani95,HillVelani99} there has been a great amount of interest in studying the Hausdorff dimension of shrinking target sets. Usually the underlying set $X$ is a fractal set, in particular when $X$ is a self-similar, self-conformal, or self-affine set and the sets $B_n$ are balls where the radius shrinks exponentially. See for example \cite[Chapter 9]{Falconer}, \cite[Chapters 2 and 4]{BishopPeres} or \cite[\S7-9]{FraserAssouad} for the relevant definitions and background on self-similar, self-conformal and self-affine sets.

In \cite{HillVelani99}, Hill and Velani studied the Hausdorff dimension of shrinking targets sets in self-conformal sets.  This was followed up more recently with work relating to the Hausdorff measures of the same sets by Baker \cite{Baker2019}, and the first author together with B\'{a}r\'{a}ny \cite{AllenBarany}. Baker also studied shrinking targets in self-similar sets with overlaps in \cite{BakerMemoirs}. The problem of shrinking targets on self-affine sets has also been studied, see \cite{BaranyRams} for the Bedford-McMullen carpets case, \cite{KoivusaloRamirezSelfAffine} for a ``typical'' self-affine set, \cite{BaranyTroscheitRandomisedTranslations} for generic (as introduced by Falconer \cite{Falconer1988}) self-affine sets, and \cite{KoivusaloLiaoRamsPathDependent} for self-affine sets where the rate of shrinking is path-dependent. Very recently, Baker and Koivusalo have studied shrinking targets in overlapping self-affine iterated function systems where the targets considered may be more exotic \cite{BakerKoivusalo}. In \cite{JordanKoivusaloPU}, the second author and Koivusalo consider shrinking targets in a class of self-affine sets. While the shrinking targets considered in \cite{JordanKoivusaloPU} are balls, the structure of the self-affine sets makes the work similar in nature to the present work where we consider rectangular shrinking targets for self-similar systems.

Much of the work above focuses on the shrinking target problem for sequences of balls. Our aim in this paper is to study the Hausdorff dimension, denoted $\dimh$, of shrinking targets in self-similar carpets where our targets are rectangles rather than balls. Results relating to rectangular shrinking targets on a restrictive class of self-affine sets have been obtained by Wang and Wu \cite[Section 12]{WangWu}, but they require the underlying self-affine set to be the Cartesian product of one-dimensional self-similar sets, see \cite{WangWu} for more details. We do not have this condition in our setup, but we do require the self-similar carpet under consideration to have a nice ``grid structure''. 

As a motivation for our setup, consider the following example. Let $V$ denote the Vicsek set or Cross Fractal. That is, {$V$ is} the set of points in $[0,1]^{2}$ whose two dimensional base $3$ expansion uses exclusively the digits $\{(0,0),(2,0),(0,2),(1,1),(2,2)\}$. Consider the transformation $\tilde{T}:[0,1]^{2}\to[0,1]^{2}$ defined by
\begin{equation*}
    \tilde{T}(x_{1},x_{2})=\left(3x_{1} \mod 1\,,\, 3x_{2} \mod 1\, \right).
\end{equation*}
Our results applied to this setting give:
\begin{theorem} \label{example result}
Let $\xi\geq \lambda>0$, {let} $y=(y_{1},y_{2})\in V$ be fixed, and consider the set 
\begin{equation*}
    V_{\lambda,\xi}(y):=\left\{x\in V : \tilde{T}^{n}(x) \in [y_{1}-3^{-n\lambda},y_{1}+3^{-n\lambda}]\times[y_{2}-3^{-n\xi},y_{2}+3^{-n\xi}] \quad \text{ for i.m. } \, \, n\in \N \right\}\, .
\end{equation*}
Suppose that the digit frequencies of $y_{2}$ exist for each digit base $3$. Then
\begin{equation*}
    \dimh {V}_{\lambda,\xi}(y) = \min \left\{ \frac{\dimh V}{1+\lambda}, \frac{1}{1+\xi}\left( \dimh V+ (\xi-\lambda)\dimh V_{y}\right)\right\},
\end{equation*}
where $V_{y}$ denotes the horizontal slice through $V$ at the point $y$.
\end{theorem}
The case where $\lambda=\xi$ is the standard shrinking target problem for shrinking targets covered by \cite{HillVelani95}, see also \cite{AllenBarany} . The case with $\xi>\lambda$ is given by our Theorem \ref{corollary} which in turn follows from Theorem \ref{main}.
See \S~\ref{statement of results} for the definition of digit frequency. As seen in \S~\ref{statement of results} the condition that the digit frequency of $y_{2}$ exists can be removed. However, what this condition does allow us to do is state our result in terms of the Hausdorff dimension of the horizontal slice. Note that our general result (Theorem~\ref{main}) gives the exact Hausdorff dimension formula for any choice of $y\in V$ in the above theorem, but the formula becomes more complicated when $y$ is not a ``typical'' point in the fractal. See \S~\ref{further remarks} for {further discussion on this matter}. \par 
Note interestingly that the dimension depends on the centre of our target. For example
\begin{equation*}
   \dimh V_{\lambda,\xi}((0,0)) = \min \left\{ \frac{\dimh V}{1+\lambda}, \frac{1}{1+\xi}\left( \dimh V + (\xi-\lambda)\dimh \cK\right)\right\}, 
\end{equation*}
where $\cK$ denotes the middle-third Cantor set, but
\begin{equation*}
    \dimh V_{\lambda,\xi}\left(\left(\tfrac{1}{2},\tfrac{1}{2}\right)\right) =  \frac{\dimh V}{1+\xi}\, .
\end{equation*}

The paper is laid out accordingly. In the following subsection we give a generalised setup in which our results are applicable. In \S~\ref{statement of results} we state our main results, firstly in a special case where the Hausdorff dimension can be stated in a intuitive manner (Theorem~\ref{corollary}) and then in the general case where the Hausdorff dimension is a little more complicated (Theorem~\ref{main}). The weakening required to state the special case, in comparison to the general case, is discussed further in \S~\ref{further remarks} where various examples are presented. In \S~\ref{prelims} and \ref{proof} we prove our main result and in \S~\ref{proof of corollary} our special case is {deduced} via the general case. \par 
We should mention that at the time of preparation of this article we discovered that Edouard Daviaud was simultaneously proving a similar result, but with different techniques. This was done via mass transference principles and can be found in \cite{Daviaudpreprint}. Theorem~3.4 in \cite{Daviaudpreprint} is essentially the same result as our Corollary \ref{corergodic}.

\subsection{Our setting}

Fix $b \in \N$ such that $b \geq 2$ and let $J \subset \{0, \dots , b-1\}^{2}$ be a proper subset with $\#J \geq 2$. For each $(u,v) \in J$, let $f_{u,v}:[0,1]^2 \to [0,1]^2$ be the map defined by
\begin{equation*}
    f_{u,v}(x,y)=\left(\frac{x+u}{b}, \frac{y+v}{b} \right).
\end{equation*}
Consider the self-similar iterated function system 
\[\Phi=\{f_{u,v} : (u,v) \in J\}\] 
and let $\Lambda$ be the attractor of $\Phi$. That is, $\Lambda$ is the unique non-empty compact subset of $[0,1]^2$ for which
\[\Lambda = \bigcup_{(u,v) \in J}{f_{u,v}(\Lambda)}.\]
The fact that such a set $\Lambda$ exists is a classical result due to Hutchinson \cite{Hutchinson}. Since all of the maps in $\Phi$ are similarities, we call $\Phi$ a \emph{self-similar iterated function system} and call $\Lambda$ a \emph{self-similar set} (or, to be more precise, the \emph{self-similar set associated to $\Phi$}).

For any $n \in \N$ and any $\bt=((u_{1},v_{1}), \dots , (u_{n},v_{n})) \in J^{n}$ we will use the shorthand notation 
\[f_{\bt}=f_{u_{1},v_{1}} \circ \dots \circ f_{u_{n},v_{n}}.\] 
For a point $(x,y) \in \Lambda$, let $f_{\bt}(x,y)_{i}$ denote the $i$th coordinate value of $f_{\bt}(x,y)$. For each $a \in \{0, \dots , b-1\}$, define
\[J_{1}(a)=\{(u,v) \in J : u=a\},\]
and    
\[J_{2}(a)=\{(u,v) \in J : v=a\}.\]
 If we consider taking the unit square $[0,1]^2$ and splitting it into a $b \times b$ grid of squares (with rows and columns indexed by $0,1,\dots,b-1$), each of the maps $f_{u,v} \in \Phi$ maps the unit square to the smaller square of side-length $b^{-1}$ in the $u$th column and $v$th row. The set $J_{2}(a)$ therefore consists of the pairs $(u,v) \in J$ corresponding to maps $f_{u,v}$ with images belonging to the $a$th row of this $b \times b$ grid. Similarly, the set $J_{1}(a)$ consists of the pairs $(u,v) \in J$ corresponding to maps $f_{u,v}$ with images belonging to the $a$th column.

For any $(x,y) \in \Lambda$ we may write
\begin{equation} \label{sequence}
    (x,y)=\left( \sum_{i=1}^{\infty}{\frac{x_{i}}{b^{i}}}, \sum_{i=1}^{\infty}{\frac{y_{i}}{b^{i}}} \right)
\end{equation}
where $(x_{i},y_{i}) \in J$ for all $i \in \N$. \emph{A priori}, the sequence $(x_{i},y_{i})_{i \in \N} \in J^{\N}$ yielding \eqref{sequence} may not be unique. Let $A(x,y)$ denote all sequences $(x_{i},y_{i})_{i \in \N} \in J^{\N}$ satisfying \eqref{sequence}. For each $(x,y) \in \Lambda$, we pick a unique representative sequence $(x_{i},y_{i})_{i \in \N}$ from $A(x,y)$ as follows: if $\#A(x,y) >1$, in the first instance choose the sequence such that
\begin{equation*}
    \prod_{i=1}^{N} \#J_{2}(y_{i}) = \max_{(x_{i}',y_{i}')_{i \in \N} \in A(x,y)} \prod_{i=1}^{N} \#J_{2}(y_{i}')  \quad \text{for all } N \geq N_{0}
\end{equation*}
for some $N_{0} \in \N$. If more than one sequence satisfies the above equality, choose the one with the most $(0,0)$ terms, then the most $(0,b-1)$ terms, and lastly the most $(b-1,0)$ terms. This ordering is not particularly important, it is simply to ensure uniqueness. Let $\bar{\Sigma}$ denote the collection of all of our unique representative sequences. Generally, the sequences in $\bar{\Sigma}$ are chosen so that, where possible, the elements of those sequences belong to the most populated row $J_{2}(a)$.
  
Define $T:\Lambda \to \Lambda$ \label{T definition} by
\begin{equation*}
    T(x,y) = T((x_1,x_2,\dots), (y_1,y_2,\dots)) =   (bx-x_1, by-y_1).
\end{equation*}
where $(x,y) \in \bar{\Sigma}$.

Throughout let $(\lambda(n))_{n \in \N}$ and $(\xi(n))_{n \in \N}$ be sequences of positive integers with $\xi(n) \geq \lambda(n)$ for all $n \in \N$ and $\lambda(n) \to \infty$ as $n \to \infty$. Let $(z,w) \in \Lambda$ and consider the shrinking target set
\begin{equation*}
    \Lambda_{\lambda,\xi}(z,w) = \left\{ (x,y) \in \Lambda \, : \, T^{n}(x,y) \in [z-b^{-\lambda(n)},z+b^{-\lambda(n)}]\times [w-b^{-\xi(n)}, w+b^{-\xi(n)}] \text{ for i.m. } \, n \in \N \right\}.
\end{equation*}
Let us write
\[J^* = \bigcup_{n=1}^{\infty}{J^n}\]
to denote the set of all possible finite (non-empty) sequences consisting of elements from $J$.
Note that $\Lambda_{\lambda,\xi}(z,w)$ is closely related to the set 
\begin{align} \label{W set}
W_{\lambda,\xi}(z,w)&:=\left\{(x,y) \in \Lambda: \left|x-f_{\bt}(z,w)_1\right|\leq b^{-\lambda(|\bt|)-|\bt|} \text{ and }\right. \nonumber \\ 
&\phantom{=========}\left. \left|y-f_{\bt}(z,w)_2\right|\leq b^{-\xi(|\bt|)-|\bt|} \text{ for i.m. } \bt \in J^*\right\},
\end{align}
which also appears in the literature, see for example \cite{AllenBarany,WangWu}. In particular, we have the following {relations between} the two sets.

\begin{proposition} \label{relationship between sets}
Fix $(z,w)\in\Lambda$. {Then,}
\begin{enumerate}[(i)]
\item If $z \not\in \{0,1\}$ and $w \not\in \{0,1\}$, then $\Lambda_{\lambda,\xi}(z,w)=W_{\lambda,\xi}(z,w)$.
\item If $z \in \{0,1\}$ or $w \in \{0,1\}$, then $\Lambda_{\lambda,\xi}(z,w) \subseteq W_{\lambda,\xi}(z,w)$ and \[W_{\lambda,\xi}(z,w) \subseteq \bigcup_{(r,s)\in\{-1,0,1\}^2} \Lambda_{\lambda,\xi}(z+r,w+s).\]
\end{enumerate}
\end{proposition}
This proposition is proved towards the end of \S~\ref{prelims}.

\section{Statements of results} \label{statement of results}

\subsection{A special case}
 We prove the following statement regarding the Hausdorff dimension of $\Lambda_{\lambda,\xi}(z,w)$.

\begin{theorem} \label{corollary}
Fix $(z,w) \in \Lambda$ and let $(z_{i},w_{i})_{i \in \N} \in \bar{\Sigma}$ be its unique representative sequence. Thus,
\begin{equation*}
(z,w)=\left( \sum_{i=1}^{\infty} z_{i}b^{-i} , \sum_{i=1}^{\infty} w_{i} b^{-i} \right).
\end{equation*} 
Suppose the limits 
\[\lambda:=\lim_{n \to \infty} \frac{\lambda(n)}{n} \quad \text{and} \quad \xi:=\lim_{n \to \infty} \frac{\xi(n)}{n}\] 
exist and that $\xi>\lambda$. Further, suppose that the limits
\begin{equation*}
    p_{a}(z,w):=\lim_{n \to \infty}\frac{\#\left\{1 \leq i \leq n : w_{i}=a \right\}}{n}
\end{equation*}
exist for each $a \in \{0, \dots, b-1\}$ and that $p_{0}(z,w)<1$ and $p_{b-1}(z,w)<1$. Then,
\begin{equation*}
    \dimh \Lambda_{\lambda,\xi}(z,w) = \min\left\{ \frac{\gamma}{1+\lambda}, \frac{1}{1+\xi}\left(\gamma + (\xi-\lambda)\gamma(z,w)_{2} \right) \right\}\, ,
\end{equation*}
where
\begin{align*}
    \gamma &=\dimh \Lambda = \frac{\log\#J}{\log b} 
\end{align*}
and
\begin{align*}
    \gamma(z,w)_{2} &=\dimh \left(\Lambda \cap \left\{(x,y) \in [0,1]^{2} : y=w\right\}\right) \\
                    &=\frac{1}{\log b} \sum_{a \in \{0, \dots , b-1\}} p_{a}(z,w) \log \#J_{2}(a)\,.
\end{align*}
\end{theorem}

\begin{remark} \rm \label{horizontal slice}
In the above, $p_{a}(z,w)$ is the limiting frequency of the digit $a$ in the word $(w_{i})_{i \in \N}$, $\gamma$ is the Hausdorff dimension of $\Lambda$, and $\gamma(z,w)_{2}$ is the Hausdorff dimension of the horizontal slice through $\Lambda$ including the point $(z,w)$. To see this, note that the horizontal slice at height~$w$ (as long as $w$ does not have $2$ different $b$-expansions  which would imply $p_0(z,w)=1$ or $p_{b-1}(z,w)=1$) can be expressed as
$$\Lambda^{w}=\left\{\sum_{i=1}^{\infty}x_ib^{-i}:(x_i,w_i)\in J\right\}.$$
Following the methods in example 1.4.2 in Bishop and Peres \cite{BishopPeres} (they cover the case $b=2$ but the general case is a simple adaptation of this) we get that
$$\dimh \Lambda^{w}=\liminf_{n\to\infty}\frac{\sum_{a=0}^{b-1}n^{-1}\#\{1\leq i\leq n:w_i=a\}\log\#J_2(a)}{\log b}.$$
Thus, in the case where the limits 
\begin{equation*}
    p_{a}(z,w):=\lim_{n \to \infty}\frac{\#\left\{1 \leq i \leq n : w_{i}=a \right\}}{n}
\end{equation*}
exist for all $0\leq a\leq b-1$ with $p_0(z,w) \neq 1$ and $p_{b-1}(z,w)\neq 1$, we get that
$$\dimh \Lambda^{w}=\frac{\sum_{a \in \{0, \dots , b-1\}} p_{a}(z,w) \log \#J_{2}(a)}{\log b}.$$
\end{remark}

\begin{remark} \rm
Theorem~\ref{corollary} follows from a more general result (Theorem \ref{main}) given in the following section. At the end of the paper we comment on what happens when the conditions in Theorem \ref{corollary} are relaxed.
\end{remark}

\begin{remark} \rm
Theorem~\ref{example result} can be deduced from Theorem~\ref{corollary} as follows. The case when $\lambda=\xi$ {corresponds to the more traditional shrinking target problem for balls and} is already known (see \cite{AllenBarany} and \cite{HillVelani95}). Thus, we may suppose $\xi>\lambda$. {Next, notice in this case that $\xi(n)=n\xi$ and $\lambda(n)=n\lambda$}, so the limits {$\lim_{n \to \infty} \frac{\lambda(n)}{n}$ and $\lim_{n \to \infty} \frac{\xi(n)}{n}$} exist. {Furthermore}, since we ask that the base $3$ digit frequencies of $y_{2}$ exist{, we have that} the limits $p_{a}(y_{1},y_{2})$ exist {for each $0 \leq a \leq 2$}. Now, it is entirely possible that we have $p_{0}(y_{1},y_{2})=1$ or $p_{2}(y_{1},y_{2})=1$. {However, in this case, this is not a problem due to the symmetry of the cross fractal; in particular, we have $\#J_{2}(0)=\#J_{2}(2)$.} See \eqref{zeros case} and \eqref{b-1 case} in \S~\ref{further remarks} for further details.
\end{remark}

\subsection{The general result}

We begin by defining the following subset of words. 

\begin{definition} \label{DefM(z,w)}
Let $(z,w) \in \Lambda$ and let $(z_{i},w_{i})_{i \in \N} \in \bar{\Sigma}$ be its unique representative sequence.
Let $(\xi(n))_{n \in \N}$ and $(\lambda(n))_{n \in \N}$ be sequences of positive integers with $\xi(n) \geq \lambda(n)$ for all $n \in \N$. For $n \in \N$, let $M_{n}(z,w)$ denote the set of words $(x_{i},y_{i})_{i \in \N} \in \bar{\Sigma}$ such that the following two conditions are met:
\begin{enumerate}[(1)]
\item{
\begin{enumerate}[(i)]
\item{$x_{i+n}=z_{i}$ for $1 \leq i < \lambda(n)$ 
{\bf \emph{or}}}
\item{there exists $1\leq j < \lambda(n)$ such that $x_{i+n}=z_{i}$ for $1\leq i \leq j-1$ and either
\begin{enumerate}[(a)] 
\item{$x_{j+n}-z_{j}=-1$ and $x_{i+n}-z_{i}=b-1$ for $j+1 \leq i < \lambda(n)$, or}
\item{$x_{j+n}-z_{j}=1$ and $z_{i}-x_{i+n}=b-1$ for $j+1 \leq i < \lambda(n)$.}
\end{enumerate}}
\end{enumerate}}
\item{
\begin{enumerate}[(i)]
\item{$y_{i+n}=w_{i}$ for $1 \leq i < \xi(n)$ {\bf \emph{or}}} 
\item{there exists $1 \leq j < \xi(n)$ such that $y_{i+n}=w_{i}$ for $1\leq i \leq j-1$ and either
\begin{enumerate}[(a)] 
\item{$y_{j+n}-w_{j}=-1$ and $y_{i+n}-w_{i}=b-1$ for $j+1 \leq i < \xi(n)$, or}
\item{$y_{j+n}-w_{j}=1$ and $w_{i}-y_{i+n}=b-1$ for $j+1 \leq i < \xi(n)$.}
\end{enumerate}}
\end{enumerate}}
\end{enumerate}
 \end{definition}
 
Essentially $M_{n}(z,w)$ contains the set of all words $(x_{i},y_{i})_{i \in \N} \in \bar{\Sigma}$ that are close to $(z_{i},w_{i})_{i \in \N}$, where closeness is determined by $n$ and the integers $\lambda(n)$ and $\xi(n)$. Generally, the larger $n$ is, the closer the points in $M_n(z,w)$ will be to $(z_i,w_i)_{i \in \N}$. Since each sequence $(x_{i},y_{i})_{i \in \N} \in \bar{\Sigma}$ maps to a unique point $(x,y) \in \Lambda$, we may often simply write "$(x,y) \in M_{n}(z,w)$", meaning that there exists a sequence $(x_{i},y_{i})_{i \in \N} \in M_{n}(z,w)$ such that $(x,y)=\Pi((x_{i},y_{i})_{i \in \N})$. Here, $\Pi: \bar{\Sigma} \to \Lambda$ denotes the usual \emph{projection mapping}, as defined in Section \ref{prelims}. For a fixed $n \in \N$ and $(z_{i},w_{i})_{i \in \N}$, there will exist a unique integer $k_{w}(n)$ with $1 \leq k_{w}(n) < \xi(n)$ such that all words $(x_{i},y_{i})_{i \in \N} \in M_{n}(z,w)$ will satisfy
 \begin{equation*}
 (x_{i+n},y_{i+n})_{i=1, \dots, \xi(n)}=(x_{i+n},w_{i})_{i=1, \dots, k_{w}(n)}(x_{j+n},y_{j+n})_{j=k_{w}(n)+1, \dots , \xi(n)}\, .
 \end{equation*}
That is, for each $(x_{i},y_{i})_{i \in \N} \in M_{n}(z,w)$, $(y_{i+n})_{i=1,\dots, k_{w}(n)}$ will be the same as $(w_{i})_{i=1,\dots,k_{w}(n)}$ irrespective of whether later digits agree or not. 

For each $a \in \{0, \dots , b-1\}$, $n \in \N$, $y \in \Lambda_{2}=\{y : (x,y) \in \Lambda \text{ for some } x\}$, and $j \geq \lambda(n)$, let
\begin{equation*}
     p_{n,j}(a,y)=\#\{ \lambda(n) \leq i \leq j : y_{i+n}=a\}
\end{equation*}
and
\begin{equation*}
    A_{n,j}(z,w)= \max_{(x,y) \in M_{n}(z,w)} \left\{ \sum_{a \in \{0, \dots , b-1\}} p_{n,j}(a, y) \log\#J_{2}(a) \right\}.
\end{equation*}

Note that for $j>\xi(n)$ there are no restrictions on the subword $(x_{i},y_{i})_{i=\xi(n)+1, \ldots, j}$ of \newline\mbox{$(x_{i},y_{i})_{i\in\N}\in M_{n}(z,w)$}, so the choice for these letters, when considering $A_{n,j}(z,w)$, would simply be the one where $\#J_{2}(a)$ is largest. For $(x_{i},y_{i})_{i \in \N} \in \bar{\Sigma}$, corresponding to a point $(x,y) \in \Lambda$, the quantity $p_{n,j}(a,y)$ tells us about the frequency of the digit $a$ in the finite subword $y_{\lambda(n)}y_{\lambda(n)+1}\dots y_{j}$ of the word $(y_{i})_{i \in \N}$. The set $M_{n}(z,w)$ is essentially constructed as follows: given a fixed $(z,w) \in \Lambda$, we approximate $z$ in the horizontal axis by the $b$-adic interval of length $b^{-\lambda(n)}$ containing $z$, and then include the two other $b$-adic intervals adjacent. We repeat this in the vertical axis by approximating $w$ by a $b$-adic interval of length $b^{-\xi(n)}$ containing $w$, and again include the two other adjacent intervals. Taking the product of these $3$ intervals in each axis we obtain $9$ rectangles with sidelengths $b^{-\lambda(n)}$ and $b^{-\xi(n)}$ in the horizontal and vertical axes respectively, and with the centre rectangle containing $(z,w)$. Then, $M_{n}(z,w)$ is the collection of all words in $\bar{\Sigma}$ with image (under the usual projection mapping) contained in the collection of rectangles constructed. To calculate the value $A_{n,j}(z,w)$ we pick the rectangle, out of those constructed in describing $M_{n}(z,w)$, with the most $b$-adic cubes (of sidelength $b^{-\xi(n)}$) that have non-empty intersection with $\Lambda$. Then $A_{n,j}(z,w)$ is the logarithm of the cardinality of the non-empty intersecting $b$-adic cubes in this rectangle. \par

We prove the following statement regarding the Hausdorff dimension of $\Lambda_{\lambda, \xi}(z,w)$.
 
\begin{theorem} \label{main}
Let 
\begin{equation*}
    s_{n}=\min\left\{ \frac{n\log\#J + A_{n,j}(z,w)}{(n+j)\log b} \, : \, \lambda(n) \leq j \leq \xi(n) \right\}\, . 
\end{equation*}
Then
\begin{equation*}
 \dimh \Lambda_{\lambda,\xi}(z,w) = \limsup_{n \to \infty} s_{n}\, . 
\end{equation*}
\end{theorem}

\section{Preliminaries} \label{prelims}

We recall the definitions of Hausdorff measure and {Hausdorff} dimension. For more details see \cite{Falconer}. For any $0 < \rho \leq \infty$, any finite or countable collection of balls $(B_i)_{i\geq 1}$ contained in $\R^{2}$ such that
$F\subseteq \bigcup_i B_i$ and $r(B_{i})<\rho$ for all $i$, is called a \emph{$\rho$-cover} of $F$. {Here we use the notation $r(B)$ to denote the radius of a ball $B$ in $\R^2$.} For $s>0$ let
\begin{equation*}
\cH_{\rho}^{s}(F)=\inf \left\{ \sum_{i} r(B_{i})^{s} : \{B_{i}\} \text{ is a $\rho$-cover of } F \right\}.
\end{equation*} 
The \textit{$s$-dimensional Hausdorff measure of $F$} is defined to be
\[
\cH^s(F)=\lim_{\rho\to 0}\cH_\rho^{s}(F).
\]
 The \textit{Hausdorff dimension} of $F$, denoted by $\dimh F$, is defined as
\begin{equation*}
\dimh F := \inf\left\{s\geq 0\;:\; \cH^s(F)=0\right\}\, .
\end{equation*}
{A fairly standard approach to computing the Hausdorff dimension of a set is to split the computations into two cases; an upper bound calculation and a lower bound calculation.} The upper bound calculation is {often} easier and can be realised by taking a standard cover of the set { while establishing lower bounds can frequently be somewhat} more {challenging. To this end,} the following \emph{mass distribution principle} {often proves to be useful (see, for example, \cite[Chapter 4.1]{Falconer}).}

\begin{lemma}[Mass Distribution Principle]\label{Mass distribution principle}
Let $\mu$ be a probability measure supported on a subset $X \subseteq \R^{2}$. Suppose that for {some} $s>0$ there exist constants $c {>0}$ {and} $\varepsilon>0$ such that
\begin{equation*} 
\mu(B) \leq c r(B)^{s}
\end{equation*}
for all open balls $B \subset \R^{2}$ with $r(B)<\varepsilon$. Then $\cH^{s}(X) \geq \frac{1}{c} {>0}$ and thus $\dimh X\geq s$.
\end{lemma}
\par

We use the notational conventions
\begin{equation*}
    \Sigma=J^{\N}, \qquad \Sigma^{n}=J^{n}, \quad \text{and} \quad \Sigma^{*}=\bigcup_{n=1}^{\infty}J^{n}.
\end{equation*}
We say that a function $f: \N \to \R$ satisfies $f(n) = o(g)$, where $g: \N \to \R_{>0}$, if $\frac{f(n)}{g(n)} \to 0$ as $n \to \infty$.
Let $\Pi:\Sigma \to \Lambda$ denote the usual projection mapping; thus,
\begin{equation*}
    \Pi((c_{i},d_{i})_{i \in \N}) = \left( \sum_{i=1}^{\infty}c_{i}b^{-i}, \sum_{i=1}^{\infty}d_{i}b^{-i} \right) \, .
\end{equation*}
Observe that $\Pi$ is injective when restricted to $\bar{\Sigma} \subset \Sigma$. For $(c,d) \in \Sigma$ and $m \in \N$, let
\begin{equation*}
    [(c_{i},d_{i})_{i=1, \dots , m}]=\left\{(u_{i},v_{i})_{i \in \N} \in \bar{\Sigma} : (u_{i},v_{i})=(c_{i},d_{i}) \text{ for } i=1, \dots , m \right\}.
\end{equation*}
 That is, $[(c_{i},d_{i})_{i =1, \dots , m}]$ is the cylinder containing all words $(x_{i},y_{i})_{i \in \N} \in \bar{\Sigma}$ with their first $m$ terms matching $(c_{i},d_{i})_{i =1, \dots , m}$. Finally, let $\sigma: \bar{\Sigma} \to \bar{\Sigma}$ denote the usual {left} shift map; that is,
\[\sigma\left( (c_{i},d_{i})_{i \in \N} \right)= (c_{i+1},d_{i+1})_{i \in \N}.\]
The following lemma follows immediately from the definition of the mapping $T$ {as given on page \pageref{T definition}}.

\begin{lemma}
For any $(c_{i},d_{i})_{i \in\N} \in\bar{\Sigma}$ we have that
\begin{equation*}
    T\circ \Pi\left((c_{i},d_{i})_{i \in \N}\right)=\Pi\circ \sigma \left( (c_{i},d_{i})_{i \in \N} \right).
\end{equation*}
\end{lemma}

The following lemma will be crucial
in the proof of Theorem \ref{main} in Section \ref{proof}. It essentially tells us that $M_{n}(z,w)$ is a good collection of words such that the corresponding cylinders cover the rectangle
\begin{equation*}
R\left((z,w), b^{-\lambda(n)}, b^{-\xi(n)}\right)=[z-b^{-\lambda(n)},z+b^{-\lambda(n)}]\times [w-b^{-\xi(n)}, w+b^{-\xi(n)}],
\end{equation*}
and are contained in the rectangle
\begin{equation*}
R\left((z,w), b^{-\lambda(n)+2}, b^{-\xi(n)+2}\right)=[z-b^{-\lambda(n)+2},z+b^{-\lambda(n)+2}]\times [w-b^{-\xi(n)+2}, w+b^{-\xi(n)+2}].
\end{equation*}
Throughout, for $(x,y) \in \R^{2}$ and $\alpha,\beta \in \R_{>0}$, we write $R((x,y),\alpha,\beta)$ to denote a rectangle with centre $(x,y)$ and sidelengths $2\alpha$ and $2\beta$.

 \begin{lemma} \label{upper}
 Let $(z,w) \in \Lambda$, let $n \in \N$, and let $M_{n}(z,w)$ be as in Definition~\ref{DefM(z,w)}.
 \begin{enumerate}[(a)]
 \item If
 \begin{equation*}
     T^{n}(x,y) \in R\left((z,w), b^{-\lambda(n)}, b^{-\xi(n)}\right),
 \end{equation*}
then $(x,y) \in M_n(z,w)$.
 \item If $(x,y) \in  M_{n}(z,w)$, then
 \begin{equation*}
     T^{n}(x,y) \in R\left((z,w), b^{-\lambda(n)+2}, b^{-\xi(n)+2}\right).
 \end{equation*}
\end{enumerate}
 \end{lemma}

 \begin{proof} \,
 
\noindent $(a)$  Writing $\|\cdot\|$ to denote the usual supremum norm on $\R^{2}$, observe that
 \begin{equation*}
     \left\|T^{n} \circ \Pi\left((x_{i},y_{i})_{i \in \N}\right)-(z,w)\right\|=\left\|\left(\sum_{i=1}^{\infty}(x_{i+n}-z_{i})b^{-i}, \sum_{i=1}^{\infty}(y_{i+n}-w_{i})b^{-i} \right) \right\|.
 \end{equation*}
Thus, if 
 \begin{equation*}
     T^{n}(x,y) \in R\left((z,w), b^{-\lambda(n)}, b^{-\xi(n)} \right),
 \end{equation*}
 then
 \begin{equation*}
     \left|\sum_{i=1}^{\infty}(x_{i+n}-z_{i})b^{-i}\right| \leq b^{-\lambda(n)} \quad \text{and} \quad \left| \sum_{i=1}^{\infty}(y_{i+n}-w_{i})b^{-i} \right| \leq b^{-\xi(n)}.
 \end{equation*}
 Consider the first coordinate axis and suppose that $(x_{i+n})_{i=1 \dots , \lambda(n)}$ does not satisfy the conditions of Definition~\ref{DefM(z,w)}. We consider the following two cases:
\begin{enumerate}[(i)] 
  \item Suppose that for some $1 \leq j<\lambda(n)$ we have $x_{i+n}=z_{i}$ for $1 \leq i \leq j-1$, but $x_{j+n}-z_{j}\geq2$. Then, {by our hypothesis combined with the reverse triangle inequality, we have}
 \begin{align*}
     \left|\sum_{i=1}^{\infty}(x_{i+n}-z_{i})b^{-i}\right| &= \left| (x_{j+n}-z_{j})b^{-j}+ \sum_{i=j+1}^{\infty}(x_{i+n}-z_{i})b^{-i}\right| \\[1ex] 
     &{\geq \left|(x_{j+n}-z_j)b^{-j}\right| - \sum_{i=j+1}^{\infty}{\left|x_{i+n}-z_i\right|b^{-i}}} \\[1ex] 
     	&\geq 2b^{-j}-\sum_{i=j+1}^{\infty}{(b-1)b^{-i}} \\[1ex]
        &= 2b^{-j}-b^{-j} \\[1ex]
        &>b^{-\lambda(n)}
 \end{align*}
and hence $T^{n}(x,y) \not \in R\left((z,w), b^{-\lambda(n)}, b^{-\xi(n)} \right)$. {A similar argument yields the same contradiction} if $x_{i+n}=z_i$ for $1 \leq i \leq j-1$ and $x_{j+n}-z_{j}\leq-2$. 
 \item Suppose that for some $1 \leq j < \lambda(n)$ we have $x_{i+n}=z_{i}$ for $1 \leq i \leq j-1$ and $x_{j+n}-z_{j}=1$, but $z_{t}-x_{t+n}\leq b-2$ for some $j+1 \leq t < \lambda(n)$. Then,
 \begin{align*} 
 \sum_{i=1}^{\infty}(x_{i+n}-z_{i})b^{-i} &= (x_{j+n}-z_{j})b^{-j}-\sum_{i=j+1}^{\infty}(z_{i}-x_{i+n})b^{-i} \\[1ex]
 &= b^{-j}-\sum_{i=j+1}^{\infty}{(z_i-x_{i+n})b^{-i}}.
 \end{align*}
 Now observe that 
 \begin{align*}
     \sum_{i=j+1}^{\infty}{(z_i-x_{i+n})b^{-i}} &\leq -b^{-t}+\sum_{i=j+1}^{\infty}(b-1)b^{-i} \\[1ex] 
     &=-b^{-t}+b^{-j}\, .
 \end{align*}
 Hence,
 \begin{align*}
\sum_{i=1}^{\infty}(x_{i+n}-z_{i})b^{-i}&\geq b^{-t}>b^{-\lambda(n)}\,
 \end{align*} 
and so $T^{n}(x,y) \not \in R\left((z,w), b^{-\lambda(n)}, b^{-\xi(n)} \right)$. The same {contradiction can be obtained via a similar argument} in the case that $x_{i+n}=z_i$ for $1 \leq i \leq j-1$, $x_{j+n}-z_{j}=-1$, and \mbox{$x_{t+n}-z_{t} \leq b-2$}.
\end{enumerate}
Analogous arguments can be applied in the case that $(y_{i+n})_{i=1,\dots,\xi(n)}$ does not satisfy the conditions of Definition \ref{DefM(z,w)}. This concludes the proof of part $(a)$ of the lemma.

\vspace{1ex}
\noindent $(b)$ Again, we start with the observation that
 \begin{equation*}
     \left\|T^{n} \circ \Pi\left((x_{i},y_{i})_{i \in \N}\right)-(z,w)\right\|=\left\|\left(\sum_{i=1}^{\infty}(x_{i+n}-z_{i})b^{-i}, \sum_{i=1}^{\infty}(y_{i+n}-w_{i})b^{-i} \right) \right\|.
 \end{equation*}
Let $T^{n}(x,y)_{1}$ denote the value of $T^{n}(x,y)$ in the horizontal axis and let $T^{n}(x,y)_{2}$ denote the value of $T^{n}(x,y)$ in the vertical axis. If $x_{i+n}=z_{i}$ for all $i=1, \dots , \lambda(n)$, then
 \begin{equation*}
|T^n(x,y)_1-z| \leq \sum_{i=\lambda(n)+1}^{\infty}(b-1)b^{-i}=b^{-\lambda(n)}
 \end{equation*}
 and so $T^{n}(x,y)_{1} \in [z-b^{-\lambda(n)},z+b^{-\lambda(n)}]$. Similarly, if $y_{i+n}=w_{i}$ for all $i=1, \dots , \xi(n)$, then
 \begin{equation*}
|T^n(x,y)_2-w| \leq \sum_{i=\xi(n)+1}^{\infty}(b-1)b^{-i}=b^{-\xi(n)}
 \end{equation*}
 and so $T^{n}(x,y)_{2} \in [w-b^{-\xi(n)},w+b^{-\xi(n)}]$. 
 
Next, we consider the case when, for some $1 \leq j < \lambda(n)$, we have $x_{i+n}=z_{i}$ for $1\leq i\leq j-1$, $x_{j+n}-z_{j}=-1$, and $x_{i+n}-z_{i}=b-1$ for all $j+1 \leq i < \lambda(n)$. In this case, we have 
 \begin{align*}
 \left|T^{n}(x,y)_{1}-z\right| &{= \left|\sum_{i=1}^{\infty}{x_{i+n}b^{-i}} - \sum_{i=1}^{\infty}{z_ib^{-i}}\right|} \\
 &{= \left| \sum_{i=j}^{\infty}{(x_{i+n}-z_i)b^{-i}} \right|} \\
 &= \left|-b^{-j}+\sum_{i=j+1}^{\lambda(n)-1}(b-1)b^{-i} +\sum_{k=\lambda(n)}^{\infty}(x_{k+n}-z_{k})b^{-k}\right| \\
 &=\left|-b^{-j} +b^{-j}-b^{-(\lambda(n)-1)} + \sum_{k=\lambda(n)}^{\infty}(x_{k+n}-z_{k})b^{-k}\right| \\
 & \leq \left|b^{-\lambda(n)+1}\right| + \sum_{k=\lambda(n)}^{\infty}|(x_{k+n}-z_k)|b^{-k} \\
 & \leq \left|b^{-\lambda(n)+1}\right| + \sum_{k=\lambda(n)}^{\infty}(b-1)b^{-k} \\
 & = 2 b^{-\lambda(n)+1} \\
 &\leq b^{-\lambda(n)+2}.
 \end{align*}
 It can be deduced similarly that if $x_{i+n}=z_i$ for all $1 \leq i \leq j-1$, $x_{j+n}-z_{j}=1$, and \mbox{$z_{i}-x_{i+n}=b-1$} for all $j+1 \leq i < \lambda(n)$, then
 \[|T^{n}(x,y)_{1}-z|\leq b^{-\lambda(n)+2}.\]
 Thus, in either case, 
 \[T^{n}(x,y)_{1} \in [z-b^{-\lambda(n)+2}, z+b^{-\lambda(n)+2}].\]  
 Similar calculations show that if $(x,y) \in M_n(z,w)$, then $T^{n}(x,y)_{2} \in [w-b^{-\xi(n)+2},w+b^{-\xi(n)+2}]$. Hence,
 \begin{equation*}
  T^{n}(x,y) \in R\left( (z,w), b^{-\lambda(n)+2}, b^{-\xi(n)+2} \right),
  \end{equation*}
  as required.
  \end{proof}

\subsection{Proof of Proposition~\ref{relationship between sets}}

{The proof of Proposition~\ref{relationship between sets} follows immediately from the following lemma.}
\begin{lemma} \label{Lemma 1}
Fix $(z,w) \in \Lambda$ and suppose $\lambda(n) { \to \infty}$ {as $n \to \infty$ (hence, by assumption, we also have $\xi(n) \to \infty$ as $n \to \infty$)}. For all sufficiently large $n \in \N$, and any $(x,y) \in \Lambda$, we have the following:
\begin{enumerate}[(i)]
\item \label{Lemma 1 i} Suppose $z \not \in \{0,1\}$ and $w \not \in \{0,1\}$. Then
\begin{align} 
T^{n}(x,y) \in [z-b^{-\lambda(n)},z+b^{-\lambda(n)}&]\times [w-b^{-\xi(n)}, w+b^{-\xi(n)}] \label{eq1} 
\end{align}
if and only if for some $\bt \in J^{n}$ we have
\begin{align} 
\left|x-f_{\bt}(z,w)_1\right|\leq &b^{-\lambda(|\bt|)-|\bt|} \quad \text{and} \quad \left|y-f_{\bt}(z,w)_2\right|\leq b^{-\xi(|\bt|)-|\bt|}. \label{eq2}
\end{align}
\item \label{Lemma 1 ii} Suppose $z \in \{0,1\}$ or $w \in \{0,1\}$. If \eqref{eq1} is satisfied,
then \eqref{eq2} holds for some $\bt \in J^{n}$.
Conversely, if \eqref{eq2} holds for some $\bt \in J^{n}$, then
\begin{equation}
T^{n}(x,y) \in \bigcup_{(r,s)\in\{-1,0,1\}^2} [z+r-b^{-\lambda(n)},z+r+b^{-\lambda(n)}]\times [w+s-b^{-\xi(n)}, w+s+b^{-\xi(n)}]. \label{eq6}
\end{equation}
\end{enumerate}
\end{lemma}

\begin{proof}
Let $(x,y) \in \Lambda$ and take $(x_i,y_i)_{i \in \N}$ to be the unique representative sequence of $(x,y)$ from $\bar{\Sigma}$. The proof of the implication $\eqref{eq1} \implies \eqref{eq2}$ is the same regardless of the values of $z$ and~$w$. The reverse implications $\eqref{eq2} \implies \eqref{eq1}$ and $\eqref{eq2} \implies \eqref{eq6}$ are similar in proof, but {differ slightly} depending on the unique representative sequence of $(x,y)$.

 Observe that \eqref{eq1} holds if and only if
\begin{equation} \label{prop equiv}
\left| \left( \sum_{i=1}^{\infty}x_{i+n}b^{-i} \right) -z \right| \leq b^{-\lambda(n)} \, \quad \text{and} \quad \, \left| \left( \sum_{i=1}^{\infty}y_{i+n}b^{-i} \right) -w \right| \leq b^{-\xi(n)},
\end{equation}
since
\begin{equation*}
T^{n}(x,y)=\left( \sum_{i=1}^{\infty}x_{i+n}b^{-i}, \sum_{i=1}^{\infty}y_{i+n}b^{-i} \right)\, .
\end{equation*}

Consider $\bt=((x_{1},y_{1}), \dots , (x_{n},y_{n})) \in J^{n}$. Then
\begin{equation*}
f_{\bt}(z,w)=\left( \left(\sum_{i=1}^{n}x_{i}b^{-i}\right)+b^{-n}z, \left(\sum_{i=1}^{n}y_{i}b^{-i}\right)+b^{-n}w \right)\,.
\end{equation*}
Note that we have 
\begin{align*}
|x-f_{\bt}(z,w)_1| &= \left|\left(\sum_{i=1}^{\infty}{x_i b^{-i}}\right)-\left(\left(\sum_{i=1}^{n}{x_{i}b^{-i}}\right)+b^{-n}z\right)\right| \\
                   &= \left|\left(\sum_{i=n+1}^{\infty}{x_{i} b^{-i}}\right)-b^{-n}z\right| \\
                   &= \left|b^{-n}\left(\sum_{i=1}^{\infty}{x_{i+n} b^{-i}}\right)-b^{-n}z\right| \\
                   &\leq b^{-n-\lambda(n)} \\
                   &= b^{-|\bt|-\lambda(|\bt|)}
\end{align*}
where the penultimate line above follows from the left-hand inequality from \eqref{prop equiv}. We can do a similar calculation in the second coordinate axis, yielding
\[|y-f_{\bt}(z,w)_2| \leq b^{-|\bt|-\xi(|\bt|)}.\] 
Hence, we have shown that  $\eqref{eq1} \implies \eqref{eq2}$ regardless of the values of $z$ and $w$. 

We now turn our attention to showing that $\eqref{eq2}$ implies $\eqref{eq1}$ when $z,w\not\in\{0,1\}$ and $\eqref{eq2}$ implies $\eqref{eq6}$ in general. We begin with showing that $\eqref{eq2}$ implies $\eqref{eq6}$. Suppose that \mbox{$(x,y)\in \Lambda$} satisfies \eqref{eq2} for some $\bt=((u_{1},v_{1}),\dots , (u_{n},v_{n})) \in J^{n}$. Let us consider for a moment the left-hand inequality from \eqref{eq2}. We have 
\begin{align}
\left| x-f_{\bt}(z,w)_{1}\right| \leq b^{-\lambda(|\bt|)-|\bt|} \quad &\implies \quad \left| \sum_{i=1}^{\infty}x_{i}b^{-i} - \left( \sum_{j=1}^{n}u_{j}b^{-{j}} + b^{-n}z \right) \right| \leq b^{-\lambda(n)-n} \nonumber\\
&\implies \quad \left| \sum_{i=1}^{n}(x_{i}-u_{i})b^{-i} +  \sum_{j=1}^{\infty}x_{n+j}b^{-(n+j)}-b^{-n}z  \right| \leq b^{-\lambda(n)-n} \nonumber\\
&\implies \quad  \left| \sum_{i=1}^{n}(x_{i}-u_{i})b^{n-i} + \sum_{j=1}^{\infty}x_{n+j}b^{-j} -z  \right| \leq b^{-\lambda(n)} \nonumber\\
&\implies \quad  \left| \sum_{i=1}^{n}(x_{i}-u_{i})b^{n-i} + \left(T^{n}(x,y)_{1}-z\right) \right| \leq b^{-\lambda(n)}. \label{eq7}
\end{align}
Note that by the reverse triangle inequality, the above yields
\begin{align}
    \left| x-f_{\bt}(z,w)_{1}\right| \leq b^{-\lambda(|\bt|)-|\bt|} &\implies \quad  \left| \sum_{i=1}^{n}(x_{i}-u_{i})b^{n-i}\right| - \left|\left(T^{n}(x,y)_{1}-z\right) \right| \leq b^{-\lambda(n)}. \label{reverse t}
\end{align}
If the absolute value of the summation appearing in {\eqref{reverse t} is} greater than or equal to $2$, {then~\eqref{reverse t}} would be false since $|T^{n}(x,y)_{1}-z|\leq1${, $b \geq 2$,} and $\lambda(n)\geq1$. So we must have that
\begin{equation} \label{summation}
\sum_{i=1}^{n}(x_{i}-u_{i})b^{n-i} \in \{-1,0,1\}.
\end{equation}
If $r=-\sum_{i=1}^{n}(x_{i}-u_{i})b^{n-i}$, then it follows from \eqref{eq7} that 
\begin{equation} \label{what we are left with}
    T^n(x,y)_1\in[z+r-b^{-\lambda(n)},z+r+b^{-\lambda(n)}],
\end{equation}
{where $r \in \{-1,0,1\}$.}
Similarly for the second coordinate axis we can show that 
$$T^n(x,y)_2\in [z+s-b^{-\lambda(n)},z+s+b^{-\lambda(n)}]$$
for some $s\in\{-1,0,1\}$. Thus we have {shown that} $\eqref{eq2}$ implies $\eqref{eq6}$.

Now we will show that $\eqref{eq2}$ implies $\eqref{eq1}$. If $r=\sum_{i=1}^{n}(x_{i}-u_{i})b^{n-i}=0$, then by \eqref{what we are left with} we are done. Thus it remains to show that $r \not \in \{-1,1\}$. In order to reach a contradiction, suppose \mbox{$r\in\{-1,1\}$}.  By the assumption appearing in Lemma~\ref{Lemma 1} $(i)$, we have that $z\not\in\{0,1\}$ and \mbox{$w\not\in\{0,1\}$}. In particular, there exists $k\in\N$ such that
\begin{equation}\label{z place}
z\in (b^{-k},1-b^{-k}).
\end{equation}
Since $\lambda(n) \to \infty$ as $n\to \infty$, there exists some $n_{0} \in \N$ such that $\lambda(n)>k$ for all $n>n_{0}$. So suppose without loss of generality that we have been working with some $n>n_{0}$. For \eqref{what we are left with} to be true we must have that 
\begin{equation*}
T^{n}(x,y)_{1}-z \in [-1-b^{-\lambda(n)},-1+b^{-\lambda(n)}] \cup [1-b^{-\lambda(n)},1+b^{-\lambda(n)}].
\end{equation*}
However, since $T^{n}(x,y)_{1} { \in [0,1]}$ {and} $z \in [0,1]$, this would imply that 
\[z \in[0,b^{-\lambda(n)}]\cup [1- b^{-\lambda(n)},1]\] 
which is false by \eqref{z place}.  Hence $r \not \in \{-1,1\}$ and so we are done. A similar calculation can be done in the second coordinate axis, thus showing that $\eqref{eq2}$ implies~$\eqref{eq1}$.
\end{proof}

 \section{Proof of Theorem \ref{main}} \label{proof}
 \subsection{Upper bound} \label{upperbound}

We begin with a proof of the upper bound of Theorem~\ref{main}. 

Let $(j(n))_{n\in\N}$ be any sequence where $\lambda(n) \leq j(n) \leq \xi(n)$ for all $n\in\N$. For each $(x,y)\in\Lambda$ let
\begin{equation*}
B_{n,j(n)}(x,y)=\left[\sum_{i=1}^{n+j(n)} x_ib^{-i},\sum_{i=1}^{n+j(n)} x_ib^{-i}+b^{-(n+j(n))}\right]\times\left[\sum_{i=1}^{n+j(n)} y_ib^{-i},\sum_{i=1}^{n+j(n)}y_ib^{-i}+b^{-(n+j(n))}\right].
\end{equation*}
We then set
\begin{equation*}
C_{n,j(n)}=\left\{B_{n,j(n)}(x,y): (x,y) \in M_{n}(z,w) \right\}.
\end{equation*}
Note  that $C_{n,j(n)}$ is a finite collection of squares with sides $b^{-(n+j(n))}$. Furthermore, by Lemma~\ref{upper}(a), we have
\begin{align*}
T^{-n}\left(R\left((z,w), b^{-\lambda(n)}, b^{-\xi(n)}\right)\right)\, &\subseteq M_n(z,w) \subseteq \bigcup_{B_{i} \in C_{n,j(n)}} B_{i}\, . 
\end{align*}
Thus, for any $N \in \N$  we have 
\begin{equation*}
     \bigcup_{n \geq N} \bigcup_{B_{i} \in C_{n,j(n)}}B_{i} \supseteq \Lambda_{\lambda,\xi}(z,w).
\end{equation*}
Hence, 
\begin{equation*}
\{B_i:B_i\in C_{n,j(n)}\text{ for some } n\geq N\}
\end{equation*}
forms a $b^{-N}$-cover of $\Lambda_{\lambda, \xi}(z,w)$. 

Observe that
 \begin{equation*}
     \#C_{n,j(n)} \leq 9(\#J)^{n} \max_{(x,y) \in M_{n}(z,w)} \left\{\prod_{i=\lambda(n)}^{j(n)}(\#J_{2}(y_{i+n})) \right\}.
\end{equation*}
 To see this, note that from the definition of $M_n(z,w)$ that there are no restrictions on the first $n$ digits of $(x_{i},y_{i})_{i \in \N}$, hence there are $(\#J)^{n}$ possibilities. Then, recalling our intuitive description of the set $M_n(z,w)$, we are considering words with images under projection sitting within $9$ specified rectangles with sidelengths $b^{-n-\lambda(n)}$ and $b^{-n-\xi(n)}$. Now, within each of these rectangles {we want to} count the number of non-empty cubes of sidelength $b^{-(n+j(n))}$ that intersect $\Lambda$. This cardinality is built up by considering the number of cubes of sidelength $b^{-n-\lambda{(n)}}$ that intersect $\Lambda$ (which is precisely $\#J_{2}(y_{n+\lambda{(n)}})$ from some rectangle with centre $(x,y)\in M_{n}(z,w)$). Then, within each of these cubes consider the number of cubes of sidelength $b^{-(n+1)-\lambda{(n)}}$ with non-empty intersection with $\Lambda$ (this is precisely $\#J_{2}(y_{n+\lambda{(n)}})\times \#J_{2}(y_{n+1+\lambda{(n)}})$). Continue this iteratively up to $j(n)-\lambda(n)$, and then take the maximum over all $(x,y)\in M_{n}(z,w)$.

For each $a \in \{0, \dots , b-1\}$, $\lambda(n) \leq j \leq \xi(n)$, and $y$ such that $(x,y) \in M_{n}(z,w)$ {for some $x$}, recall that
 \begin{equation*}
    p_{n,j}(a, y)=\#\{\lambda(n) \leq i \leq j : y_{i+n}=a \}
 \end{equation*}
 and
 \begin{equation*}
   A_{n,j}(z,w)=\max_{(x,y) \in M_{n}(z,w)} \left\{ \sum_{a \in \{0, \dots , b-1\}} p_{n,j}(a, y) \log\#J_{2}(a) \right\}.   
 \end{equation*}
Moreover, note that {we may rewrite}
\begin{equation*}
 A_{n,j}(z,w)=\log \left( \max_{(x,y) \in M_{n}(z,w)} \left\{\prod_{i=\lambda(n)}^{j}(\#J_{2}(y_{i+n})) \right\} \right)\, .
 \end{equation*}
Also, recall that we define
\begin{equation*}
    s_{n}=\min\left\{ \frac{n\log\#J + A_{n,j}(z,w)}{(n+j)\log b} \, : \, \lambda(n) \leq j \leq \xi(n) \right\}\, . 
\end{equation*}
For each $n \in \N$, choose $\lambda(n) \leq j(n) \leq \xi(n)$ such that the minimum value $s_{n}$ is attained; that is, $j(n)$ is chosen such that
 \begin{equation*}
     s_{n}=\frac{n\log\#J+A_{n,j(n)}(z,w)}{(n+j(n))\log b}.
 \end{equation*}
Next, let $s=\limsup_{n \to \infty} s_{n}$ {and let $\varepsilon > 0$. Choose $N_{\varepsilon} \in \N$ such that for any $n \geq N_{\varepsilon}$ we have $s_n \leq s+\frac{\varepsilon}{2}$. Let 
\[s^+_{\varepsilon} = \max\left(\{s_n : 1 \leq n \leq N_{\varepsilon}\} \cup \{s\}\right) \qquad \text{and} \qquad j_{\varepsilon}=\max\left\{\frac{j(n)}{n}: 1 \leq n \leq N_{\varepsilon}\right\}.\]}

Then, {by our choice of $s$,} we have

 \begin{align*}
     \cH_{b^{-N}}^{s{+\varepsilon}}(\Lambda_{\lambda,\xi}(z,w)) &\leq \sum_{n \geq N} \sum_{B_{i} \in C_{n,j(n)}} b^{-(n+j(n))(s+\varepsilon)} \\
     & \leq \sum_{n\geq N} \# C_{n,j(n)} b^{-(n+j(n))(s+\varepsilon)}\\
     & \leq 9\sum_{n\geq N} (\#J)^{n} \max_{(x,y) \in M_{n}(z,w)} \left\{\prod_{i=\lambda(n)}^{j(n)}(\#J_{2}(y_{i+n})) \right\}  b^{-(n+j(n))(s+\varepsilon)}\\
     &= 9\sum_{n \geq N} b^{\frac{n\log \#J}{\log b}} b^{\frac{A_{n,j(n)}(z,w)}{\log b}} b^{-(n+j(n))(s+\varepsilon)} \\
     &{= 9\sum_{n \geq N} b^{(n+j(n))\left(\frac{n\log \#J + A_{n,j(n)}(z,w)}{(n+j(n))\log b}\right)} b^{-(n+j(n))(s+\varepsilon)}} \\
     &{= 9\sum_{n \geq N} b^{(n+j(n))s_n} b^{-(n+j(n))(s+\varepsilon)}} \\
     &{= 9 \sum_{n=N}^{N_{\varepsilon}}{b^{(n+j(n))(s_n-s-\varepsilon)}} + 9 \sum_{n=N_{\varepsilon}+1}^{\infty}{b^{(n+j(n))(s_n - s - \varepsilon)}}} \\
     &{\leq 9 \sum_{n=N}^{N_{\varepsilon}}{b^{(n+j(n))(s^{+}_{\varepsilon}-s-\varepsilon)}} + 9 \sum_{n=N_{\varepsilon}+1}^{\infty}{b^{(n+j(n))\left(s+\frac{\varepsilon}{2}-s-\varepsilon \right)}}} \\
     &{\leq 9 \sum_{n=N}^{N_{\varepsilon}}{b^{n(1+j_{\varepsilon})(s^{+}_{\varepsilon}-s)}} + 9 \sum_{n=N_{\varepsilon}+1}^{\infty}{b^{-n\frac{\varepsilon}{2}}}} \\
     &{\leq 9 \sum_{n=1}^{N_{\varepsilon}}{b^{n(1+j_{\varepsilon})(s^{+}_{\varepsilon}-s)}} + 9 \sum_{n=N_{\varepsilon}+1}^{\infty}{b^{-n\frac{\varepsilon}{2}}}}\, .
 \end{align*}
 Since this bound is independent of $N$ it follows by letting $N\to \infty$ that \mbox{$\cH^{s+\varepsilon}(\Lambda_{\lambda,\xi}(z,w))<\infty$} and so by the definition of Hausdorff dimension it follows that \mbox{$\dimh \Lambda_{\lambda,\xi}(z,w) \leq s$}. {This completes the proof of} the upper bound {of Theorem \ref{main}}.
 
 \subsection{Lower bound} \label{lowerbound}
  
 \subsubsection{Outline of the strategy}
 For establishing the lower bound in Theorem \ref{main} we use the mass distribution principle (Lemma~\ref{Mass distribution principle}). Firstly we translate the problem of considering points in $\Lambda_{\lambda,\xi}(z,w)$ to the symbolic space. We then construct a suitable measure on the symbolic space such that the projection of the support of the measure is a Cantor set contained in a subset of $\Lambda_{\lambda,\xi}(z,w)$. We then use the mass distribution principle to calculate a lower bound for the Hausdorff dimension of $\Lambda_{\lambda,\xi}(z,w)$ by computing the H\"{o}lder exponent of the measure at an arbitrary point in the Cantor set.
  
 \subsubsection{Construction of a measure}
 Let $(n_{k})_{k \in \N}$ be a strictly increasing sequence of integers such that:
 \begin{enumerate}[(i)]
     \item There exists a fixed $\Delta \in (1, \infty)$ \label{Delta def} such that for all $k \geq 1$ we have
 \begin{equation*}
     n_{k+1}> \Delta\left(\sum_{i=1}^{k}(\xi(n_{i})+2)\right) \quad \text{and} \quad n_{k+1}>n_{k}+\xi(n_{k})+2.
 \end{equation*}
 \item {For some $\varepsilon>0$ we have
 \begin{equation*} 
 s_{n_{k}} \geq \limsup_{n \to \infty} s_{n} - \varepsilon 
 \end{equation*}
for all $k \in \N$.}
 \end{enumerate}
For each $k \in \N$, define 
 \begin{equation*}
     \Lambda_{\lambda,\xi}((z,w),n_{k})=\left\{ (x,y) \in \Lambda : T^{n_{k}}(x,y) \in R\left((z,w),b^{-\lambda(n_k)+2}, b^{-\xi(n_{k})+2} \right) \right\}.
 \end{equation*}
As above, choose $\lambda(n_{k})\leq j(n_{k}) \leq \xi(n_{k})-1$ such that $j(n_k)$ satisfies
 \begin{equation*}
 \frac{n_{k}\log\#J+A_{n_{k},j(n_k)}(z,w)}{(n_{k}+j(n_k))\log b}=\min\left\{ \frac{n_{k}\log\#J+A_{n_{k},j}(z,w)}{(n_{k}+j)\log b} \, : \, \lambda(n_{k}) \leq j \leq \xi(n_{k})-1 \right\}\, .
 \end{equation*}
Furthermore, for each $k \in \N$, pick some word 
\[(\hat{x}(k),\hat{y}(k))=\Pi((\hat{x}_{i}(k), \hat{y}_{i}(k))_{i \in \N}) \in M_{n_{k}}(z,w)\] 
such that
\begin{equation*}
\prod_{i=\lambda(n_{k})}^{j(n_k)} (\#J_{2}(\hat{y}_{i+n_{k}}(k)))= \max_{(x,y) \in M_{n_{k}}(z,w)} \left\{ \prod_{i=\lambda(n_{k})}^{j(n_k)} (\#J_{2}(y_{i+n_{k}}(k))) \right\}.
\end{equation*}

Next, we define a measure $\mu$ on the cylinders 
\[\{[(u_{i},v_{i})_{i=1,\dots, n}]: n\in \N \text{ and } (u,v) \in \bar{\Sigma}\}.\] 
To begin with, for $1 \leq m \leq n_1$, we assign cylinders corresponding to words of length $m$ mass $\frac{1}{(\#J)^m}$; that is, 
\[\mu([(u_i, v_i)_{i=1,\dots,m}])= \frac{1}{(\#J)^m}.\]
In general, for a cylinder $[(u_i,v_i)_{i=1,\dots,m+1}]$ corresponding to a word of length $m+1 \geq n_1+1$, we assign it a portion of $\mu([(u_i,v_i)_{i=1,\dots,m}])$ as follows:
\begin{itemize}
\item{If $n_{k}+\xi(n_{k})+2< m+1 \leq n_{k+1}$ for any $k \in \N$, then distribute the mass evenly over the digits so that
\[\mu([(u_i,v_i)_{i=1,\dots,m+1}]) = \frac{\mu([(u_i,v_i)_{i=1,\dots,m}])}{\#J}.\]}
\item{If $n_{k} < m+1 \leq n_{k}+ \lambda(n_{k})+2$ for any $k \in \N$, then assign all the mass at each digit to the digit matching $(\hat{x}_{m+1}(k),\hat{y}_{m+1}(k))_{k\in\N}$ so that
\begin{equation*}
    \mu([(u_i,v_i)_{i=1,\dots,m+1}]) =
    \begin{cases}
      \mu([(u_i,v_i)_{i=1,\dots,m}]) &\textrm{if } (u_{m+1},v_{m+1})=(\hat{x}_{m+1}(k),\hat{y}_{m+1}(k)) \;, \\[1ex]
      0 &\textrm{otherwise}.
    \end{cases}
\end{equation*}}
\item{If $n_{k}+\lambda(n_{k})+2< m+1 \leq n_{k}+\xi(n_{k})+2$ for any $k \in \N$, then distribute the mass at each digit evenly over the digits in the set $J_{2}(\hat{y}_{m+1}(k))$ so that
\begin{equation*}
    \mu([(u_i,v_i)_{i=1,\dots,m+1}]) =
    \begin{cases}
      \frac{\mu([(u_i,v_i)_{i=1,\dots,m}])}{\#J_{2}(\hat{y}_{m+1}(k))} &\textrm{if } (u_{m+1},v_{m+1})\in J_{2}(\hat{y}_{m+1}(k)) \;, \\[1ex]
      0 &\textrm{otherwise}.
    \end{cases}
\end{equation*}}
\end{itemize}

Essentially, the measure $\mu$ distributes the mass evenly over the first $n_{1}$ digits in the construction set for $\Lambda$, then all mass is concentrated to the digits that match up with the word whose cylinder projects (by $\Pi$) to a ball containing the most populated rectangle formed in the definition of $M_{n}(z,w)$. When the sidelength of this projected cylinder is slightly smaller than the longest sidelength of the rectangle the measure $\mu$ then begins to distribute the mass over digits whose words have cylinders that project to a populated ball in said rectangle. This continues until the projected cylinder's sidelength is slightly smaller than the smallest sidelength of the rectangle. The cycle then begins again, i.e. the mass is evenly distributed over the digits in the construction set of $\Lambda$ until we reach the $n_{2}$th digit etc.  \par 
 Define $\nu=\mu \circ \Pi^{-1}$. It follows from the construction of the measure $\mu$ that
 \begin{equation*}
     \nu\left( \bigcap_{k\in \N}\Lambda_{\lambda,\xi}((z,w), n_{k}) \right)=1.
 \end{equation*}
 
 \subsubsection{Measure of a general ball}
 Fix $(x,y) \in \bigcap_{k \in \N}\Lambda_{\lambda,\xi}((z,w),n_{k})$ and let $r>0$. We will consider the measure $\nu(B((x,y),r))$ of the ball $B((x,y),r)$. To this end, suppose that $n \in \N$ is the unique integer such that
 \begin{equation*}
     b^{-n-1}<r \leq b^{-n}
 \end{equation*}
and notice that $B((x,y),r)$ can intersect at most $9$ projections (in terms of the projection mapping $\Pi$) of $n$-level cylinders. Thus 
 \begin{equation*}
     \nu(B((x,y),r)) \leq 9 \mu([(x_{i},y_{i})_{i=1, \dots , n}]).
 \end{equation*}
Let $k \in \N$ be the unique value of $k$ such that $n_{k}<n\leq n_{k+1}$. For convenience, let us also define
 \begin{equation*}
     V(n) = \prod_{a \in \{0, \dots , b-1\}}(\#J_{2}(a))^{-p^{*}_{n_{k},n-n_{k}}(a,\hat{y}(k))},
 \end{equation*}
 where
 \begin{equation*}
 p^{*}_{n_{k},n-n_{k}}(a,\hat{y}(k))=\#\{\lambda(n_k) +3 \leq i \leq \min\{n-n_k,\xi(n_k)+2\} \, : \, \hat{y}_{i+n_k}(k)=a\}.
 \end{equation*}

By the construction of the measure $\mu$ and choice of $(x,y)$ with unique representative sequence \mbox{$(x_{i},y_{i})_{i\in\N} \in \bar{\Sigma}$}, we have the following non-zero possibilities for the measure of a cylinder $\mu([(u_{i},v_{i})_{i=1, \dots , n}])$ which may be intersected by $B((x,y),r)$:
\begin{itemize}
\item{If $n_{k}< n \leq n_{k}+\lambda(n_{k})+2$, then 
\begin{equation*}
     \mu([(u_{i},v_{i})_{i=1, \dots , n}]) =(\#J)^{-\left( n_{k}-\sum_{i=1}^{k-1}(\xi(n_{i})+2) \right)}\prod_{l=1}^{k-1}\prod_{i=\lambda(n_{l})+3}^{\xi(n_{l})+2}(\#J_{2}(\hat{y}_{i+n_{l}}(l)))^{-1}=:U(n_{k}).
\end{equation*}}
\item{If $n_{k}+\lambda(n_{k})+2<n \leq n_{k}+\xi(n_{k})+2$, then  
\begin{equation*}
     \mu([(u_{i},v_{i})_{i=1, \dots , n}])= U(n_{k})\prod_{i=\lambda(n_{k})+3}^{n-n_{k}}(\#J_{2}(\hat{y}_{i+n_{k}}(k)))^{-1} = U(n_{k})V(n).
\end{equation*}}
\item{If $n_{k}+\xi(n_{k})+2<n \leq n_{k+1}$, then
\begin{equation*}
     \mu([(u_{i},v_{i})_{i=1, \dots , n}])=U(n_{k})V(n_{k}+\xi(n_{k})+2)(\#J)^{-(n-(n_{k}+\xi(n_{k})+2))}.
\end{equation*}}
\end{itemize}

Notice that \label{T bound}
\begin{align*}
U(n_{k}) &= (\#J)^{-\left( n_{k}-\sum_{i=1}^{k-1}(\xi(n_{i})+2) \right)}\prod_{l=1}^{k-1}\prod_{i=\lambda(n_{l})+3}^{\xi(n_{l})+2}\left(\#J_{2}(\hat{y}_{i+n_{l}}(l))\right)^{-1} \\[1ex] 
     &\leq (\#J)^{-(n_{k}-\sum_{i=1}^{k-1}(\xi(n_{i})+2))} \\[1ex]
     &\leq (\#J)^{-n_{k}(1-\Delta^{-1})} \\[1ex]
     &=b^{-n_{k}(1-\Delta^{-1})\frac{\log\#J}{\log b}},
\end{align*} 
where $\Delta$ is as on page \pageref{Delta def}.

Next, recall that
 \begin{equation*}
     p_{n,j}(a,\hat{y}(k))=\#\{\lambda(n) \leq i \leq j \, : \, \hat{y}_{i+n}(k)=a \}\, , 
 \end{equation*}
and let
\[C=\left(\max_{a \in \{0, \dots , b-1\}}\#J_{2}(a)\right)^{3}.\] 
Counting the number of times each digit $a\in\{0,\dots,b-1\}$ appears in the string of digits $\left(\hat{y}_{n_{k}+i}(k)\right)_{i=\lambda(n_{k}),\ldots, n-n_{k}}$ we get 
\[\prod_{a \in \{0, \dots , b-1\}}(\#J_{2}(a))^{-p_{n_{k},n-n_{k}}(a,\hat{y}(k))}.\] This is essentially what $V(n)$ does except for that $V(n)$ starts counting digits at $n_{k}+\lambda(n_{k})+3$. Thus,
 \begin{equation*}
 C^{-1}\prod_{a \in \{0, \dots , b-1\}}(\#J_{2}(a))^{-p_{n_{k},n-n_{k}}(a,\hat{y}(k))} \leq  V(n) \leq C \prod_{a \in \{0, \dots , b-1\}}(\#J_{2}(a))^{-p_{n_{k},n-n_{k}}(a,\hat{y}(k))}\,.
 \end{equation*}
 
Putting all of the above together, and recalling that $b^{-n-1} < r \leq b^{-n}$ and $n_{k} < n \leq n_{k+1}$, we have the following cases:
\begin{itemize}
\item If $n_{k}<n \leq n_{k}+\lambda(n_{k})+2$, then
\begin{align*}
\frac{\log \nu(B((x,y),r))}{\log r} &\geq \frac{-\log U(n_{k})}{n\log b} \\[1ex] 
                                    &\geq  \frac{n_{k}(1-\Delta^{-1})\log\#J}{n\log b} \\[1ex]
                                    &\geq \frac{n_{k}(1-\Delta^{-1})\log\#J}{(n_k + \lambda(n_k) + 2)\log b}.
\end{align*}
\item If $n_{k}+\lambda(n_{k})+2<n \leq n_{k}+\xi(n_{k})+2$, then 
\begin{align*}
\frac{\log \nu(B((x,y),r))}{\log r} &\geq \frac{-\log U(n_{k}) - \log V(n)}{n\log b} \\ 
&\geq  \frac{n_{k}(1-\Delta^{-1})\log\#J + A_{n_{k},n-n_k}(z,w) -\log C}{n\log b}.
\end{align*}
\item If $n_{k}+\xi(n_{k})+2<n \leq n_{k+1}$, then 
\begin{align*}
&\frac{\log \nu(B((x,y),r))}{\log r} \\
&\phantom{=====}\geq \frac{-\log U(n_{k}) - \log V(n_{k}+\xi(n_{k})+2) - \log(\#J)^{-(n-(n_{k}+\xi(n_{k})+2))}}{n\log b}\\[1ex]
&\phantom{=====} \geq \frac{(n_{k}(1-\Delta^{-1})+n-(n_{k}+\xi(n_{k})+2) )\log\#J + A_{n_{k}, \xi(n_{k}) +2}(z,w) -\log C}{ n \log b} \\[1ex]
&\phantom{=====} \geq \frac{(n_{k}(1-\Delta^{-1})+(n_k + \xi(n_k) + 2)-(n_{k}+\xi(n_{k})+2) )\log\#J + A_{n_{k}, \xi(n_{k}) +2}(z,w) -\log C}{ (n_k + \xi(n_k) + 2) \log b} \\[1ex]
&\phantom{=====} \geq \frac{(n_{k}(1-\Delta^{-1}))\log\#J + A_{n_{k}, \xi(n_{k}) +2}(z,w) -\log C}{ (n_k + \xi(n_k) + 2) \log b}.
\end{align*}
\end{itemize} 
Thus, noting that the endpoints of the second case correspond to the lower bounds for the first and third of the above cases and recalling that $\Delta > 1$, we have that 
 \begin{align*}
\frac{\log \nu(B((x,y),r))}{\log r} &\geq \min \left\{  \frac{n_{k}(1-\Delta^{-1})\log\#J + A_{n_{k},j}(z,w)-\log C}{(n_{k}+j)\log b} \quad : \, \lambda(n_{k})+2<j \leq \xi(n_{k})+2 \right\} \\[1ex]
                                    &\geq \min \left\{  \frac{n_{k}(1-\Delta^{-1})\log\#J + (1-\Delta^{-1})A_{n_{k},j}(z,w)-\log C}{(n_{k}+j)\log b} \quad : \, \lambda(n_{k})+2<j \leq \xi(n_{k})+2 \right\}.
 \end{align*}
 Since $\lim_{k\to\infty}\frac{-\log C}{n_{k}\log b}= 0$ and $\lim_{k\to\infty} s_{n_k}=\limsup_{n\to\infty} s_n$,  
 we have for all sufficiently large $k\in\N$ (which corresponds to all sufficiently small $r>0$) that
 \begin{align*}
     \frac{\log \nu(B((x,y),r))}{\log r} &\geq \min \left\{  \frac{(1-\Delta^{-1})\left(n_{k}\log\#J + A_{n_{k},j}(z,w)\right)}{(n_{k}+j)\log b} \quad : \, \lambda(n_{k})+2<j \leq \xi(n_{k})+2 \right\}- \varepsilon \\
     &= (1-\Delta^{-1})s_{n_{k}}-\varepsilon \\
     & \geq (1-\Delta^{-1})\limsup_{n\to \infty} s_{n} -2\varepsilon.
 \end{align*}
 Since this holds for all $\varepsilon> 0$ and all $\Delta\in (1,\infty)$, by the mass distribution principle (Lemma~\ref{Mass distribution principle}) we have that
 \begin{equation*}
     \dimh \Lambda_{\lambda,\xi}(z,w) \geq \limsup_{n \to \infty} s_{n}. 
 \end{equation*}

  \section{Proof of Theorem~\ref{corollary}} \label{proof of corollary}
Throughout this section we take the assumptions of Theorem \ref{corollary}. That is:
\begin{enumerate}[(i)]
\item{the limits 
\[\lambda:=\lim_{n \to \infty} \frac{\lambda(n)}{n} \quad \text{and} \quad \xi:=\lim_{n \to \infty} \frac{\xi(n)}{n}\] 
exist with $\xi>\lambda$, and} 
\item{we have a fixed 
\[(z,w)=\left(\sum_{i=1}^{\infty} z_ib^{-i},\sum_{i=1}^{\infty} w_ib^{-i}\right) \in \Lambda\] 
where for each $0\leq a\leq b-1$ the limits
$$p_a(z,w)=\lim_{n\to\infty}\frac{\#\{1\leq i\leq n:w_i=a\}}{n}$$
exist with $p_0(z,w)<1$ and $p_{b-1}(z,w)<1$.}
\end{enumerate}

Recall, in Section 2.2, for $(z,w)\in \Lambda$ and $n\in\N$ we defined $k_{w}(n)$ to be the minimum number of terms for which $(y_{i+n})_{i\in\N}$ and $(w_{i})_{i\in\N}$ must agree in order to have $(x,y)\in M_{n}(z,w)$ for some $x$. That is, $k_{w}(n)$ is the unique integer, with $1 \leq k_{w}(n) < \xi(n)$, such that all words $(x_{i},y_{i})_{i \in \N} \in M_{n}(z,w)$ will be of the form
 \begin{equation*}
 (x_{i+n},y_{i+n})_{i=1, \dots, \xi(n)}=(x_{i+n},w_{i})_{i=1, \dots, k_{w}(n)}(x_{j+n},y_{j+n})_{j=k_{w}(n)+1, \dots , \xi(n)}\, .
 \end{equation*}
 Furthermore, it follows from Definition~\ref{DefM(z,w)} that if $k_{w}(n)<\xi(n)-1$ then
 \begin{equation*}
     (w_{i})_{i=k_{w}(n)+2,\ldots,\xi(n)-1}=(0)_{i=k_{w}(n)+2,\ldots, \xi(n)-1} \, ,\quad \text{ or } \quad (w_{i})_{i=k_{w}(n)+2,\ldots,\xi(n)-1}=(b-1)_{i=k_{w}(n)+2,\ldots, \xi(n)-1}\, .
 \end{equation*}
 For fixed $(z,w)\in \Lambda$ consider the sequence $(k_{w}(n))_{n\in\N}$. Since $p_{0}(z,w)<1$ and $p_{b-1}(z,w)<1$ we can deduce the following:

\begin{lemma} \label{freq}
We have that 
\begin{equation*}
    \lim_{n\to\infty}\frac{k_{w}(n)}{n}=\xi\, .
\end{equation*}
\end{lemma}
\begin{proof}
Suppose $k_{w}(n)<\xi(n)-3$, otherwise the statement is immediate, and suppose $(w_{i})_{i=1,\ldots, \xi(n)-1}$ ends in a string of zeros. We have that for each $n\in\N$
\begin{equation}\label{case0}
\#\{1\leq i\leq\xi(n)-1:w_i=0\}=\#\{1\leq i\leq k_{w}(n)+1:w_i=0\}+(\xi(n)-1)-k_{w}(n)-1.
\end{equation}
Thus since
$$\#\{1\leq i\leq\xi(n)-1:w_i=0\}=p_0(z,w)(\xi(n)-1)+\mathit{o}(n)= p_0(z,w)\xi(n)+\mathit{o}(n),$$
we have that
\begin{equation*}
\xi(n)(1-p_0(z,w))+\mathit{o}(n)=k_{w}(n)+2-\#\{1\leq i\leq k_{w}(n)+1:w_i=0\}\, ,
\end{equation*}
and so
\begin{equation*}
\xi(n)(1-p_0(z,w))+\mathit{o}(n)=k_{w}(n)-\#\{1\leq i\leq k_{w}(n)+1:w_i=0\}\, .
\end{equation*}
Since we also have that
\begin{equation*}
\#\{1\leq i\leq k_{w}(n)+1:w_i=0\}=p_0(z,w)k_{w}(n)+\mathit{o}(n),
\end{equation*}
we see that 
\[\frac{k_{w}(n)}{n} = \frac{\xi(n)}{n} + o(1).\]
Thus, taking the limit as $n$ tends to infinity we see that
\[\lim_{n \to \infty}{\frac{k_w(n)}{n}} = \xi\]
as claimed. The proof works similarly in the case that $(w_{i})_{i=1,\ldots, \xi(n)-1}$ ends in a string of $(b-1)$s.
\end{proof}

We will also use the following estimate:

\begin{lemma}\label{shiftedfrequency}
For all $a \in \{0,1,\dots, b-1\}$ we have that   
$$\#\{\lambda(n) \leq i \leq j : w_{i}=a \}=(j-\lambda(n))p_a(z,w)+\mathit{o}(n)$$ 
where $\lambda(n)\leq j\leq \xi(n)$.
\end{lemma}
\begin{proof}
Let $a \in \{0,1,\dots,b-1\}$ and note that 
\[\#\{1\leq i\leq n:w_i=a\}=np_a(z,w)+\mathit{o}(n).\] Thus
\begin{eqnarray*}
\#\{\lambda(n) \leq i \leq j : w_{i}=a \}&=&\#\{1\leq i\leq j:w_i=a\}-\#\{1\leq i \leq \lambda(n)-1:w_i=a\}\\
&=&(j-\lambda(n)+1)p_a(z,w)+\mathit{o}(j-\lambda(n))
\\
&=&(j-\lambda(n))p_a(z,w)+\mathit{o}(n)\, ,
\end{eqnarray*}
where we have used that $\lambda(n)\leq j \leq \xi(n)$, $\lambda(n)=\lambda n+\mathit{o}(n)$ and $\xi(n)=\xi n+\mathit{o}(n)$ to make the error dependent exclusively on $n$.
\end{proof}
Recall that for $n \in \N$ and $j \geq \lambda(n)$, we define
\begin{equation*}
    A_{n,j}(z,w)= \max_{(x,y) \in M_{n}(z,w)} \left\{ \sum_{a \in \{0, \dots , b-1\}} p_{n,j}(a, y) \log\#J_{2}(a) \right\}.
\end{equation*} 
Combining the previous two results we see that:
\begin{lemma}\label{keybound}
For $n\in\N$ and $\lambda(n)\leq j\leq\xi(n)$ we have that
$$A_{n,j}(z,w)=\left((j-\lambda(n))\sum_{a\in \{0,\ldots, b-1\}} p_a(z,w)\log\#J_2(a)\right)+\mathit{o}(n)$$
\end{lemma}
\begin{proof}
Let $(z,w) \in \Lambda$ be fixed, let $n\in\N$, and let $k_{w}(n)$ be as above. For each $y$ where $(x,y) \in M_{n}(z,w)$ for some $x$ we have
\begin{equation*}
(y_{i+n})_{i=1, \dots , \xi(n)}=(w_{i})_{i=1, \dots , k_{w}(n)}(y_{i+n})_{i=k_w(n)+1, \dots , \xi(n)}.
\end{equation*}
 It follows that
\begin{align*}
A_{n,j}(z,w)& =\sum_{a \in \{0,\dots, b-1\}} \#\left\{\lambda(n) \leq i \leq \min\{k_{w}(n),j\} : w_{i}=a \right\}\log\#J_{2}(a) \, \\
& \quad \quad + \, \, \max_{(x,y) \in M_{n}(z,w)} \left\{ \sum_{a \in \{0, \dots , b-1\}} \#\left\{ \min\{k_{w}(n),j\} < i \leq j : y_{i+n}=a \right\}\log\#J_{2}(a) \right\} \, .
\end{align*}
We consider the two parts of this sum separately. Let us define \[D=\max_{a \in \{0,1,\dots,b-1\}}\log \#J_2(a).\] 
By Lemma \ref{freq} we have that $k_{w}(n) = n \xi + o(n)$. Since we also have $j\leq \xi(n)$, it follows that
$$
\max_{(x,y) \in M_{n}(z,w)} \left\{ \sum_{a \in \{0, \dots , b-1\}} \#\{ \min\{k_{w}(n),j\} < i \leq j : y_{i+n}=a\}\log\#J_{2}(a)\right\}
\leq D(\xi(n)-k_{w}(n))=\mathit{o}(n).$$

On the other hand, again using Lemma~\ref{freq} and combining this with Lemma \ref{shiftedfrequency} we get that
\begin{align*}
&\sum_{a \in \{0,\dots, b-1\}} \#\left\{\lambda(n) \leq i \leq \min\{k_{w}(n),j)\} : w_{i}=a\right\}\log\#J_{2}(a)\\
&=\sum_{a \in \{0,\dots, b-1\}} \#\left\{\lambda(n) \leq i \leq j : w_{i}=a \right\}\log\#J_{2}(a)
-\sum_{a \in \{0,\dots, b-1\}}\#\left\{\min\{k_{w}(n),j\} < i\leq j : w_{i}=a\right\}\log\#J_{2}(a)\\
&=(j-\lambda(n))\sum_{a\in \{0,\ldots,b-1\}} p_a(z,w)\log\#J_2(a)+\mathit{o}(n)
\end{align*}
and the result follows.
\end{proof}

\begin{proof}[Proof of Theorem \ref{corollary}]
To prove Theorem~\ref{corollary}, let $\{s_{n}\}_{n \in \N}$ be the sequence of real numbers defined by Theorem~\ref{main} so
\begin{equation*}
\limsup_{n \to \infty} s_{n}=\dimh \Lambda_{\lambda,\xi}(z,w)\, .
\end{equation*}
For each $n \in \N$, let $u(n)$ be an integer with $\lambda(n)\leq u(n) \leq \xi(n)$ such that 
\begin{equation*}
s_n=\frac{n\log\#J+A_{n,u(n)}(z,w)}{(n+u(n))\log b}.
\end{equation*}
We now choose a subsequence $(n_i)_{i\in \N}$ such that
\[\lim_{i\to\infty}s_{n_i}=\limsup_{n\to\infty}s_n=\dimh \Lambda_{\lambda,\xi}(z,w)\] 
and for which
\[\lim_{i \to \infty} \frac{ u(n_{i}) }{ n_{i} }=u\in [\lambda,\xi].\] 
We can do this as $\lim_{n\to\infty}\tfrac{\lambda(n)}{n}=\lambda$ and $\lim_{n\to\infty}\tfrac{\xi(n)}{n}=\xi$ so $\left(\frac{u(n)}{n}\right)_{n\in\N}$ is a bounded sequence with limit points in $[\lambda,\xi]$. 
Applying Lemma \ref{keybound} gives that  
\begin{eqnarray*}
\limsup_{n\to\infty}s_n&=&\lim_{i\to\infty}s_{n_i}\\
&=&\lim_{i\to\infty}\frac{n_i\log\#J+( u(n_i)-\lambda(n_{i}))\sum_{a=0}^{b-1} p_a(z,w)\log\#J_2(a)+\mathit{o}(n_i)}{(n_i+u(n_i))\log b}\\
&=&\frac{1}{1+u}(\gamma+(u-\lambda)\gamma(z,w)_2).
\end{eqnarray*}
We have that for $u$ between $\lambda$ and $\xi$ the function 
\[u\mapsto\frac{1}{1+u}(\gamma+(u-\lambda)\gamma(z,w)_2)\] is monotonic and hence
$$\dimh\Lambda_{\lambda,\xi}(z,w)=\limsup_{n\to\infty}s_n\geq \min\left\{\frac{\gamma}{1+\lambda},\frac{1}{1+\xi}(\gamma+(\xi-\lambda)\gamma(z,w)_2)\right\}.$$

In the other direction we have
\begin{align*}
s_n&=\min\left\{\frac{n\log\#J+A_{n,j}(z,w)}{(n+j)\log b}:\lambda(n)\leq j\leq\xi(n)\right\} \\
   &\leq\min\left\{\frac{n\log\#J}{(n+\lambda(n))\log b},\frac{n\log\#J+A_{n,\xi(n)}(z,w)}{(n+\xi(n))\log b}\right\}.
\end{align*} 
Note that
$$\lim_{n\to\infty}\frac{n\log\#J}{(n+\lambda(n))\log b}=\frac{\gamma}{1+\lambda}.$$ 
Furthermore, by Lemma \ref{keybound}, we also have
$$\lim_{n\to\infty}\frac{n\log\#J+A_{n,\xi(n)}(z,w)}{(n+\xi(n))\log b}=\frac{1}{1+\xi}\left(\gamma+(\xi-\lambda)\gamma(z,w)_2\right)\, .$$
Thus,
$$\dimh\Lambda_{\lambda,\xi}(z,w)\leq \min\left\{\frac{\gamma}{1+\lambda},\frac{1}{1+\xi}\left(\gamma+(\xi-\lambda)\gamma(z,w)_2\right)\right\}$$
and the proof of Theorem \ref{corollary} is complete. 
\end{proof}

\section{Further remarks} \label{further remarks}
In this section we consider some examples illustrating why the assumptions in Theorem~\ref{corollary} are necessary and explain how they may be weakened in certain circumstances. 

\subsection{Digit frequencies: $p_{0}(z,w) <1$ and $p_{b-1}(z,w)<1$}

We start by looking at the assumption that $p_{0}(z,w)<1$ (the case where $p_{b-1}(z,w)<1$ is analogous). This assumption is used to show that words $(x,y)\in M_n(z,w)$ will satisfy that for $\lambda(n)<j\leq\xi(n)$,
\begin{equation*}
p_{n,j}(a,y)=p_{a}(z,w)(j-\lambda(n))+\mathit{o}(n) \quad \text{ for all $0\leq a\leq b-1$.}
\end{equation*}
So we will now look at the case where $p_{0}(z,w)=1$, but all the other assumptions of Theorem~\ref{corollary} are met, in particular \[\lim_{n\to\infty}\frac{\lambda(n)}{n}=\lambda<\xi = \lim_{n\to\infty}\frac{\xi(n)}{n}.\]

If $w=0$, then it will be the case that \mbox{$p_{0}(z,w)=1$}. Moreover, in this case, we have that $(x,y)\in M_{n}(z,w)$ implies $y_{i+n}=w_i=0$ for all $1\leq i<\xi(n)$. Thus, in this case, for $\lambda(n)\leq j \leq \xi(n)$ we will have that 
\[p_{n,j}(a,y)=p_{a}(z,w)(j-\lambda(n)+1)+o(n)\] for all $0\leq a\leq b-1$ (where the $o(n)$ term is coming from the possibility that we may have $y_{n+\xi(n)} \neq w_{\xi(n)}$). In turn, this means in this case that 
\[A_{n,j}(z,w)=(j-\lambda(n)+1)\log{\#J_2(0)} + o(n)\] and hence
\[s_n = \min\left\{\frac{n \log{\#J} + (j-\lambda(n)+1)\log{\#J_2(0)}}{(n+j)\log{b}}: \lambda(n) \leq j \leq \xi(n)\right\} + o(1),\]
where $s_n$ is as defined in Theorem \ref{main}.
Thus, when the other assumptions of Theorem~\ref{corollary} are satisfied, it follows from Theorem \ref{main} that
\begin{equation} \label{zeros case} 
\dimh\Lambda_{\lambda,\xi}(z,w)=\min\left\{\frac{\gamma}{1+\lambda},\frac{1}{1+\xi}\left(\gamma+(\xi-\lambda)\frac{\log \#J_2(0)}{\log b}\right)\right\}.
\end{equation}

Following an analogous argument, when $w = 1$ (i.e. $w_i = b-1$ for all $i \in \N$), we get
\begin{equation} \label{b-1 case} 
\dimh\Lambda_{\lambda,\xi}(z,w)=\min\left\{\frac{\gamma}{1+\lambda},\frac{1}{1+\xi}\left(\gamma+(\xi-\lambda)\frac{\log \#J_2(b-1)}{\log b}\right)\right\}.
\end{equation}

This means we have the following corollary to Theorem \ref{main} regarding typical points for an ergodic measure.

\begin{corollary}\label{corergodic}
Let $\mu$ be a $T$-invariant ergodic probability measure with $\mu(\Lambda)=1$. If we have \[\lambda=\lim_{n\to\infty}\frac{\lambda(n)}{n} \qquad \text{and} \qquad \xi=\lim_{n\to\infty}\frac{\xi(n)}{n}\] 
exist, then for $\mu$-almost all $(z,w)\in\Lambda$ we have that
\begin{equation*}
\dimh \Lambda_{\lambda,\xi}(z,w)=\min\left\{\frac{\gamma}{1+\lambda},\frac{1}{1+\xi}\left(\gamma+(\xi-\lambda)\gamma_2(\mu)\right)\right\} \, ,
\end{equation*} 
where
\[\gamma_2(\mu)=\frac{\sum_{a=0}^{b-1} p_a(\mu)\log \#J_2(a)}{\log b}\]
with $p_{a}(\mu)$ as defined in \eqref{frequency}.
\end{corollary}

\begin{proof}
Let $\mu$ be a $T$-invariant ergodic probability measure with $\mu(\Lambda)=1$ (see Definitions~1.1 and 1.4 in~\cite{Waltersbook} for the relevant definitions). First of all we show that for $\mu$-almost all $(z,w)$, $w$ has a unique base $b$-expansion. To see this we consider the sets
\begin{equation*}
    A=\left\{\left(z,\tfrac{i}{b}\right)\in\Lambda:i\in\{1,\ldots,b-1\}\right\}\quad \text{ and } \quad B=\left\{\left(z,\tfrac{i}{b}\right)\in\Lambda:i\in\{0,b\}\right\}\, .
\end{equation*}
Note that $w$ does not have a unique base $b$-expansion if and only if $T^n(z,w)\in A$ for some $n\in \N \cup \{0\}$ (note that if $(z,w) \in B$, then $w=0=(0,0,\dots)$ or $w=1=(b-1,b-1,\dots)$ and, in particular, has a unique base $b$-expansion). Furthermore, if $(z,w)\in A$, then $T(z,w)\in B$. We also have 
$$A\cup B = T^{-1}(B)$$
which means that $\mu(A\cup B)=\mu(T^{-1}(B))=\mu(B)$. Since $A$ and $B$ are disjoint this gives $\mu(A)=0$ and so 
\begin{equation*}
\mu\left(\bigcup_{n=0}^{\infty}T^{-n}(A)\right)=0\, .    
\end{equation*}
Thus, since $\mu$ is ergodic, it follows by the Birkhoff Ergodic Theorem (see Theorem 1.14 in~\cite{Waltersbook} and the remark straight after) that for $\mu$-almost all $(z,w)$
and for each $0\leq a\leq b-1$ we have that 
\begin{equation}\label{frequency}
p_a(z,w)=\mu\left(\left\{(x,y)\in\Lambda:y\in \left[\tfrac{a}{b},\tfrac{(a+1)}{b}\right]\right\}\right)=:p_a(\mu)\, .
\end{equation}
This means that for $\mu$-almost all $(z,w)$ we have
\begin{equation*}
    \gamma(z,w)_2=\frac{\sum_{a=0}^{b-1} p_a(\mu)\log \#J_2(a)}{\log b}=\gamma_2(\mu)\, ,
\end{equation*}
where $\gamma(z,w)_2$ is as seen in Theorem 2.1.
Finally note that if $p_0(\mu)=1$ then $w=0$ for \mbox{$\mu$-almost} all $(z,w)$ and if $p_{b-1}(\mu)=1$ then $w=1$ for $\mu$-almost all $(z,w)$. Thus, either the assumptions for Theorem \ref{corollary} are met or $w \in \{0,1\}$ and so we are in the case dealt with above. Thus the desired result follows.
\end{proof}

In the case where $w=(w_i)_{i\in\N}$ has $w_i\neq 0$ for infinitely many $i \in \N$ but where \mbox{$p_{0}(z,w)=1$}, it is possible that Theorem \ref{corollary} will still hold. If we have that $\#J_2(b-1)\leq\#J_2(0)$, then 
\[A_{n,j}(z,w)=(j-\lambda(n)+1)\log\#J_2(0) + o(n)\] 
and so 
$$\dimh\Lambda_{\lambda,\xi}(z,w)=\min\left\{\frac{\gamma}{1+\lambda},\frac{1}{1+\xi}\left(\gamma+(\xi-\lambda)\frac{\log \#J_2(0)}{\log b}\right)\right\}.$$ 
In the case where $\#J_2(b-1)>\#J_2(0)$, for $a\in \{0,\ldots,b-1\}$ let
$$k_n(w,a)=\inf\left\{j-n:w_j\neq a \text{ and } j>n\right\}\, .$$
If $\lim_{n\to\infty}\frac{k_n(w,0)}{n}=0$ and $(x,y)\in M_n(z,w)$ then
\begin{equation*}
    \left| p_{n,j}(0,y)-\#\left\{ \lambda(n)\leq i \leq j : w_{i}=0\right\}\right|=\mathit{o}(n)\, . 
\end{equation*}
Thus
$$A_{n,j}(z,w)=(j-\lambda(n)+1)\log\#J_2(0)+\mathit{o}({n})$$
and so we can again conclude from Theorem \ref{main}  that
$$\dimh\Lambda_{\lambda,\xi}(z,w)=\min\left\{\frac{\gamma}{1+\lambda},\frac{1}{1+\xi}\left(\gamma+(\xi-\lambda)\frac{\log \#J_2(0)}{\log b}\right)\right\}.$$ 
Similarly, if the sequence $(w_{i})_{i\in\N}$ has $w_{i}\neq b-1$ for infinitely many $i\in\N$ but $p_{b-1}(z,w)=1$, and either $\#J_{2}(b-1)\geq \#J_{2}(0)$ or $\lim_{n\to \infty} \tfrac{k_{n}(w,b-1)}{n}=0$, then
\begin{equation*}
    \dimh\Lambda_{\lambda,\xi}(z,w)=\min\left\{\frac{\gamma}{1+\lambda},\frac{1}{1+\xi}\left(\gamma+(\xi-\lambda)\frac{\log \#J_2(b-1)}{\log b}\right)\right\}.
\end{equation*}
Now consider the case where the sequence $(w_{i})_{i\in\N}$ has $w_{i}\neq 0$ for infinitely many $i\in\N$ and 
\begin{equation*}
    p_{0}(z,w)=1 \, , \quad \#J_{2}(b-1)> \#J_{2}(0)\, , \quad \text{ and } \quad  \limsup_{n\to \infty} \tfrac{k_{n}(w,0)}{n}>0\, .
\end{equation*}
Here the outcome can be different. For example, suppose that $J_2(0)$ is a proper subset of $J_2(b-1)$ and $w$ is chosen such that
$$\limsup_{n\to\infty}\frac{k_n(w,0)}{n}=\eta>\xi-\lambda.$$
In this case we can find infinitely many values $(n_k)_{k\in\N}$ where there exists $g(n_{k})<\lambda(n_k)$ with $w_{g(n_{k})}\neq 0$ and $w_i=0$ for all $g(n_{k})< i\leq \xi(n_k)$. Since we are assuming $J_2(b-1)\supset J_2(0)$ this will mean that we can find $(z,y)\in M_{n_{k}}(z,w)$ where $y_{g(n_k)}=w_{g(n_k)}-1$ and $y_i=b-1$ for all $g(n_k)< i\leq \xi(n_k)$. This means that we will have 
$$A_{n_{k},j}(z,w)=(j-\lambda(n_{k})+1)\log\#J_2(b-1)$$
for all $\lambda(n_{k})<j\leq \xi(n_{k})$ for every $k\in\N$ and hence, by Theorem~\ref{main}, we have
$$\dimh\Lambda_{\lambda,\xi}(z,w)=\min\left\{\frac{\gamma}{1+\lambda},\frac{1}{1+\xi}\left(\gamma+(\xi-\lambda)\frac{\log \#J_2(b-1)}{\log b}\right)\right\}.$$ 
Thus, if
\[\frac{\gamma}{1+\lambda}>\frac{1}{1+\xi}\left(\gamma+(\xi-\lambda)\frac{\log \#J_2(0)}{\log b}\right),\] 
then we have that
$$\dimh\Lambda_{\lambda,\xi}(z,w)>\min\left\{\frac{\gamma}{1+\lambda},\frac{1}{1+\xi}\left(\gamma+(\xi-\lambda)\frac{\log \#J_2(0)}{\log b}\right)\right\}$$ 
since $J_2(b-1) \supset J_2(0)$. Note that this is the value given by Theorem \ref{corollary}. 

\subsection{Limits: $p_{a}(z,w)$ always exists}

We now turn to the assumption that the limits $p_a(z,w)$ always exist. We will still assume that the limits
\[\lim_{n\to\infty}\frac{\lambda(n)}{n}=\lambda \qquad \text{and} \qquad \lim_{n\to\infty}\frac{\xi(n)}{n}=\xi\]
exist and $\lambda < \xi$. Recall that
$$ s_n=\min\left\{\frac{n\log\#J+A_{n,j}(z,w)}{(n+j)\log{b}}:\lambda(n)\leq j\leq\xi(n)\right\}.$$
Without the assumption that the limits $p_{a}(z,w)$ exist we do not necessarily have that
$$\limsup_{n\to\infty} s_n=\min\left\{\frac{\gamma}{1+\lambda},\frac{1}{1+\xi}\left(\gamma+(\xi-\lambda)\gamma(z,w)_{2} \right)\right\}$$
where
\begin{equation*}
    \gamma(z,w)_{2} =\dimh \left(\Lambda \cap \left\{(x,y) \in [0,1]^{2} : y=w\right\}\right)\, .
\end{equation*}
 In particular, the conclusion of Theorem \ref{corollary} does not necessarily hold in this case.

To see this let $b=3$ and $J=\{(0,0),(2,0),(0,2)\}$. We take a sequence of positive integers $(n_j)_{j\in\N}$ where 
\[n_1=1 \qquad \text{and} \qquad \lim_{j\to\infty}\frac{n_{j+1}}{n_j}=\infty.\]
We then pick \[(z,w)=\left(\sum_{i=1}^{\infty}z_ib^{-i},\sum_{i=1}^{\infty}w_ib^{-i}\right)\] where $(z_i,w_i)=(0,0)$ if $j$ is odd and $n_j\leq i<n_{j+1}$ and $(z_i,w_i)=(0,2)$ if $j$ is even and \mbox{$n_j\leq i<n_{j+1}$}. 

Note if the limits $p_{a}(z,w)$ do not exist, then we cannot write \begin{align*}
    \gamma(z,w)_{2} &=\dimh \left(\Lambda \cap \left\{(x,y) \in [0,1]^{2} : y=w\right\}\right) \\
                    &=\frac{1}{\log b} \sum_{a \in \{0, \dots , b-1\}} p_{a}(z,w) \log \#J_{2}(a)\,.
\end{align*} 
In keeping with Remark~\ref{horizontal slice}, by considering the sequence along $(n_{2j+1})_{j\in \N}$, we have that
\begin{align*}
    \dimh \left(\Lambda \cap \left\{(x,y) \in [0,1]^{2} : y=w\right\}\right)& = \liminf_{n\to \infty}\frac{\sum_{a=0}^{2}\frac{\#\{1\leq i\leq n: w_{i}=a\}}{n}\log\#J_2(a)}{\log 3} \\
    &=\liminf_{n\to \infty} \frac{\#\{1\leq i\leq n: w_{i}=0\}}{n}\frac{\log 2}{\log 3}\\
    &=0\, .
\end{align*}
 So
$$\min\left\{\frac{\gamma}{1+\lambda},\frac{1}{1+\xi}\left(\gamma+(\xi-\lambda)\gamma(z,w)_{2} \right)\right\}=\frac{\gamma}{1+\xi}.$$
However,
\begin{align*}
\dimh \Lambda_{\lambda,\xi}(z,w) &=\limsup_{n\to\infty}s_n =\lim_{j\to\infty} s_{n_{2j}} \\
&=\min\left\{\frac{\gamma}{1+\lambda},\frac{\gamma+(\xi-\lambda)\left(\frac{\log 2}{\log 3}\right)}{1+\xi}\right\} \\
&>\frac{\gamma}{1+\xi}.
\end{align*}
Moreover, it is also possible to choose $(z,w)$ where it is not possible to just consider the cases $j=\lambda(n)$ and $\xi(n)$ when computing the minimum for $s_n$. In such cases we need to use the more general statement given by Theorem \ref{main}.

\subsection{Convergence: $\left(\frac{\lambda(n)}{n}\right)$ and $\left(\frac{\xi(n)}{n}\right)$ converge}

Finally we consider the case where the limits $p_a(z,w)$ exist for all $0\leq a \leq b-1$ with \mbox{$p_0(z,w) \neq 1$} and $p_{b-1}(z,w)\neq 1$ but where the sequences $\left(\frac{\lambda(n)}{n}\right)$ and $\left(\frac{\xi(n)}{n}\right)$ may not be convergent. By adapting the arguments in the proof of Theorem \ref{corollary} we get that
$$\dimh \Lambda_{\lambda,\xi}(z,w)=\limsup_{n\to\infty}\min\left\{\frac{\gamma}{1+\left(\tfrac{\lambda(n)}{n}\right)},\frac{\gamma+\left(\tfrac{\xi(n)-\lambda(n)}{n}\right)\gamma(z,w)_2}{1+\left(\tfrac{\xi(n)}{n}\right)}\right\}.$$

\noindent{\bf Acknowledgements.} For the purpose of open access, the authors have applied a Creative Commons Attribution (CC BY) licence to any Author Accepted Manuscript version arising from this submission. The third name author is funded by a Leverhulme Trust Early Career Research fellowship ECF-2024-401.

\bibliographystyle{abbrv}
\bibliography{biblio}
 
 \end{document}